\documentclass[a4paper, notitlepage]{article}
\usepackage[latin1]{inputenc}
\usepackage[english, british]{babel}
\usepackage{calc,amssymb,mathtools}
\mathtoolsset{mathic}
\usepackage{amsthm}
\usepackage{stmaryrd,units,enumitem,tabularx,booktabs,multicol}
\usepackage{hyperref}
\hypersetup{
breaklinks=true,
colorlinks=true,
linkcolor=black,
anchorcolor=black,
citecolor=black,
filecolor=black,
menucolor=black,
pagecolor=black,
urlcolor=black
}
\usepackage[alphabetic,non-sorted-cites]{amsrefs}
\usepackage[all]{xy}
\chardef\bslash=`\\ 

\theoremstyle{plain} 
   \newtheorem{theorem}{Theorem}[section]
   \newtheorem*{theorem*}{Theorem}
   \newtheorem{corollary}[theorem]{Corollary}
   \newtheorem*{corollary*}{Corollary}
   \newtheorem{lemma}[theorem]{Lemma}
   \newtheorem*{lemma*}{Lemma}
   \newtheorem{proposition}[theorem]{Proposition}
   \newtheorem*{proposition*}{Proposition}
\theoremstyle{definition}
   
   \newtheorem*{definition*}{Definition}
\theoremstyle{remark}
   \newtheorem{example}[theorem]{Example}
   \newtheorem*{example*}{Example}
   \newtheorem{remark}[theorem]{Remark}
   \newtheorem*{remark*}{Remark}
\newcommand{\Z}{\mathbb{Z}}
\newcommand{\R}{\mathbb{R}}
\newcommand{\C}{\mathbb{C}}

\DeclareMathOperator{\id}{id}
\DeclareMathOperator{\rank}{rank}
\DeclareMathOperator{\cd}{cd}
\renewcommand{\rank}{\mathrm{rk}}

\newcommand{\quotient}[2]{\raisebox{3pt}{\ensuremath{#1}}\!\Big/\raisebox{-3pt}{\ensuremath{#2}}}

\newcommand{\eps}{\varepsilon}
\renewcommand*{\bar}[1]{\smash{\overline{#1}}}
\renewcommand*{\tilde}[1]{\smash{\widetilde{#1}}}

\newcommand{\et}{\text{et}}

\newcommand{\nondeg}{\ensuremath{\text{nd}}}

\newcommand{\topl}{\text{top}}

\DeclareMathOperator{\im}{im}

\DeclareMathOperator{\coker}{coker}

\newcommand{\cc}[1]{#1_\bullet}

\DeclareMathOperator{\Thom}{Thom}
\DeclareMathOperator{\point}{point}

\newcommand{\sheaf}[1]{\mathcal{#1}}
\DeclareMathOperator{\CH}{CH}

\DeclareMathOperator{\Pic}{Pic}
\DeclareMathOperator{\Jac}{Jac}

\newcommand{\A}{\mathbb{A}}

\renewcommand{\P}{\mathbb{P}}
        
\newcommand{\OO}{\mathcal{O}}

\newcommand*{\longtwoheadrightarrow}{\ensuremath{\relbar\joinrel\twoheadrightarrow}}
\newcommand{\mm}[1]{\left(\begin{smallmatrix}#1\end{smallmatrix}\right)}
\newenvironment{smalldiagram}{\begin{equation*}\SelectTips{cm}{10}}{\end{equation*}}
\newcommand{\dtriangle}[6]{
  \xymatrix@R=9pt{
  {#1}\ar[rr]^-{#2} && {#3} \ar[dl]^{#4}\\
  & {#5}\ar@{..>}[ul]^{#6}
  }
  }
\newcolumntype{M}{>{$}c<{$}}

\entrymodifiers={+!!<0pt,\fontdimen22\textfont2>}
\SelectTips{cm}{10}
\newcommand{\ie}{\mbox{i.\thinspace{}e.\ }}
\newcommand{\eg}{\mbox{e.\thinspace{}g.\ }}
\newcommand{\cf}{\mbox{c.\thinspace{}f.\ }}
\DeclareMathOperator{\graded}{gr}
\newcommand{\gr}{\graded}

\DeclareMathOperator{\Sq}{Sq}

\DeclareMathOperator{\Kgroup}{K}
        \newcommand{\K}{\ensuremath{\Kgroup}}
\DeclareMathOperator{\KOgroup}{KO}

        \newcommand{\KO}{\ensuremath{\KOgroup}}
        \newcommand{\rK}{\ensuremath{\widetilde{\K}{}} }
        \newcommand{\Kh}{\ensuremath{\KOgroup}}
        
        \newcommand{\rKO}{\ensuremath{\widetilde{\KO}{}} }
        \newcommand{\KOK}[3][]{\dfrac{{\KO_{#1}^{#2}}{#3}}{{\K_{#1}}{#3}}}
        \newcommand{\tKOK}[3][]{(\KO^{#2}_{#1}\!\!/\K)_{#1}{#3}}
        
        \newcommand{\trKOK}[2]{(\rKO^{#1}\!/{\rK}){#2}}

    \DeclareMathOperator{\Wgroup}{W}
        \newcommand{\W}{\Wgroup}
        \newcommand{\rW}{\ensuremath{\widetilde{\W}{}} }
        \DeclareMathOperator{\GWgroup}{GW}
        \newcommand{\GW}{\GWgroup}
        \newcommand{\rGW}{\ensuremath{\widetilde{\GW}{}} }
        \newcommand{\w}{\mathit{w}}
        \newcommand{\gw}{\mathit{gw}}

\newcommand{\lb}[1]{{\mathcal{#1}}}
\newcommand{\vb}[1]{{\mathcal{#1}}}
\newcommand{\dual}{\vee}

\long\def\symbolfootnote[#1]#2{\begingroup%
\def\thefootnote{\fnsymbol{footnote}}\footnote[#1]{#2}\endgroup}

\usepackage{geometry}
\usepackage{enumerate}
\usepackage{needspace}
\newcommand*{\sectionair}{0.5\textheight}
\newcommand*{\subsectionair}{0.2\textheight}
\newcommand*{\encouragepagebreak}[1]{\needspace{#1}}
\newcommand{\shortendisplayskip}{\vspace{-.5\baselineskip}}
\newcommand{\GWorW}{{(Grothendieck\discretionary{-)}{}{-)}Witt} }
\newcommand{\GrothendieckWitt}{Grothen\-dieck-Witt }
\begin{document}
\title{Witt Groups of Curves and Surfaces}
\author{Marcus Zibrowius%
 \thanks{Bergische Universit^^e4t Wuppertal,
  Gau{\ss}stra{\ss}e 20,
  42119 Wuppertal,
  Germany}
}
\date{}
\maketitle
\begin{abstract}

We study Witt groups of smooth curves and surfaces over algebraically closed fields of characteristic not two. In both dimensions, we determine both the classical Witt group and Balmer's shifted Witt groups. In the case of curves, the results are supplemented with a complete description of the (shifted) \GrothendieckWitt groups.

In a second step, we analyse the relationship of Witt groups of smooth complex curves and surfaces with their real topological K-groups. They turn out to be surprisingly close:  for all curves and for all projective surfaces of geometric genus zero, the Witt groups may be identified with the quotients of their even KO-groups by the images of their complex topological K-groups under realification.
\end{abstract}
\begin{quote}
  \tableofcontents
\end{quote}
\thispagestyle{empty}
\enlargethispage{2cm}
\bigskip
\setlength{\parindent}{0pt}
\addtolength{\parskip}{3pt}
\addtolength{\topsep}{3pt}

\encouragepagebreak{\sectionair}
\section*{Introduction}
The study of Witt groups of varieties was initiated by Knebusch in the 1970's \cite{Knebusch:varieties}. Despite Arason's early success with the calculation of the Witt groups of projective spaces over arbitrary fields (of characteristic not two) \cite{Arason}, concrete computations remain challenging to this day, and known results generally depend on the chosen ground field. In dimensions one and two, such results include explicit descriptions of the Witt groups of smooth curves and surfaces over the complex numbers due to Fern{\'a}ndez-Carmena \cite{Fernandez} and over the reals due to Knebusch and Sujatha \citelist{\cite{Knebusch:curves}\cite{Sujatha:RPS}}, as well as structural results over finite and local fields \citelist{\cite{Parimala:curves-local}\cite{Auel:Clifford-invariant}}.

Since Knebusch's original definition, the theoretical framework surrounding Witt groups has undergone considerable developments. In particular, by the work of Balmer on Witt groups of triangulated categories, we now know that the classical Witt group \(\W(X)=\W^0(X)\) of a variety naturally fits into a four-periodic family of ``shifted'' Witt groups \(\W^i(X)\) \citelist{\cite{Balmer:TWGI}\cite{Balmer:TWGII}}. Moreover, it has become clear that these groups in turn fit into the larger framework of Karoubi's hermitian algebraic K-theory and should be interpreted as hermitian K-groups of negative degrees \cite{Schlichting:GWnotes}. Under this interpretation, the hermitian K-groups of degree zero correspond to Balmer and Walter's \GrothendieckWitt groups \(\GW^i(X)\) \cite{Walter:TGW}. Generally speaking, the focus of attention seems to be shifting gradually from Witt to \GrothendieckWitt groups; see for example the recent papers \cite{Fasel:IntersectionFormula}, \cite{FaselSrinivas:Chow-Witt}, \cite{Schlichting:MayerVietoris} or \cite{Levine:Slice-and-GW}.

The purposes of this paper are twofold. Firstly, we extend Fern{\'a}ndez-Carmena's computations of the classical Witt groups of a smooth complex curve or surface \(X\) in several directions:
\begin{itemize}
\item We compute all of Balmer's shifted Witt groups.
\item We relax the condition on the ground field \(k\)---we require only \(\mathrm{char}(k) \neq 2\) and \mbox{\(\cd_2(k(X)) \leq 2\)}.  This condition is satisfied for example by curves over finite fields (\cf Remark~\ref{rem:cd-assumptions}).
\item In the case of curves over an algebraically closed field of characteristic not two, we include complete computations of the \GrothendieckWitt groups:
\end{itemize}
\begin{theorem}
Let \( C \) be a smooth curve over an algebraically closed field of characteristic not two. Let \(H^i_{\et}(C;\Z/2)\) denote its ^^e9tale cohomology groups with \(\Z/2\)-coefficients, and let \(\Pic(C)\) denote its Picard group. The \GrothendieckWitt and Witt groups of \( C \) are as follows:
\begin{align*}
   &\GW^0(C) = \Z \oplus H^1_{\et}(C;\Z/2)\oplus H^2_{\et}(C;\Z/2)  && \W^0(C) = \Z/2 \oplus H^1_{\et}(C;\Z/2)  \\
   &\GW^1(C) = \Pic(C)                                           && \W^1(C) = H^2_{\et}(C;\Z/2)                    \\
   &\GW^2(C) = \Z                                                && \W^2(C) = 0                                \\
   &\GW^3(C) = \Z/2 \oplus \Pic(C)                               && \W^3(C) = 0
\end{align*}
\end{theorem}
This result is given in Theorem~\ref{thm:W_Curve} below, which also includes a description of the \GrothendieckWitt and Witt groups twisted by a non-trivial line bundle over \(C\). A similar description of the shifted Witt groups of smooth surfaces is presented in Theorem~\ref{thm:W_Surface}.
\thispagestyle{empty}
\enlargethispage{1cm}

Secondly, we analyse the relationship of the \GrothendieckWitt and Witt groups of a smooth complex curve or surface \(X\) with the real topological K-groups of its complex points \(X(\C)\) equipped with the analytic topology. Here, we use the comparison homomorphisms
\begin{alignat*}{3}
   \gw^i\colon{}&& \GW^i(X)&\rightarrow \KO^{2i}(X(\C))\\
   \w^i\colon{}&&  \W^i(X)&\rightarrow \KOK{2i}{(X(\C))}
\end{alignat*}
introduced in \cite{Me:WCCV} under certain mild hypotheses concerning the representability of hermitian K-theory (\cite{Me:WCCV}*{Standing Assumptions 1.9}). As shown there, these maps are isomorphisms whenever $X$ is cellular. There is, of course, no reason for this result to hold in much greater generality. Nonetheless, the following comparison result is achieved in Theorems~\ref{thm:comparison_Curves} and \ref{thm:comparison_Surfaces}:

\begin{theorem}
When \(X\) is a smooth complex curve, the comparison maps \( \gw^i \) are surjective and the maps \( \w^i \) are isomorphisms. When \(X\) is a smooth complex surface, the same claim holds if and only if every continuous complex line bundle over \( X \) is algebraic, \ie if and only if the natural map \( \Pic(X)\rightarrow H^2(X;\Z) \) is surjective. (For projective surfaces, this condition is equivalent to the condition that the geometric genus \( h^{2,0} \) of \( X \) be zero.)
\end{theorem}

Our computations follow a route first described by Totaro \cite{Totaro:Witt}, which is based on a series of spectral sequences developed over the past years by Balmer, Walter and Pardon. Although Fern{\'a}ndez-Carmena managed to obtain \(\W^0(X)\) without explicit use of this machinery, the approach described here seems significantly more transparent even for this classical case.  The only place where the assumption on \(\cd_2(k(X))\) enters into the computation is the evaluation of the Bloch-Ogus spectral sequence for ^^e9tale cohomology.  Thus, it seems plausible that similar techniques can be used to compute Witt groups over more general ground fields. The \GrothendieckWitt groups of curves are obtained from their Witt groups via Karoubi's fundamental exact sequence.

Our proof of the comparison result relies on the aforementioned explicit calculations. Firstly, it is easy to see that the groups $\W^i(X)$ and $\KO^{2i}(X)/\K^0(X)$ of smooth complex varieties of dimension at most two are abstractly isomorphic in the cases claimed. For $i=0$, we can deduce that the isomorphism is induced by the map $w^0$ from the fact that all elements in $\W^0(X)$ are detected by the first two Stiefel-Whitney classes. The shifted groups $\W^1(X)$ and $\W^2(X)$ require more work. To obtain the assertions concerning the comparison maps, we decompose an arbitrary surface into a union of curves and an affine piece whose Picard group vanishes modulo two. This is the only step into which the mentioned hypotheses concerning hermitian K-theory enter (\cf Remark~\ref{rem:disclaimer}).

The structure of this article is as follows. In the first section, we review some of the basic aspects of Witt groups and real topological K-groups. We also include a brief account of the theories of Stiefel-Whitney classes in both settings and explain in how far they coincide for complex varieties. Section~\ref{sec:CS:Curves} deals with the \GrothendieckWitt groups of smooth curves. The main calculations of Witt groups and of the groups \( \tKOK{2i}{(X)} \) are presented in Section~\ref{sec:CS:Surfaces}, while the comparison result is derived in Section~\ref{sec:CS:comparison}. Finally, we briefly analyse how this result is related to the Quillen-Lichtenbaum conjecture and its analogue for hermitian K-theory, which has recently appeared in \cite{BKOS:QuillenLichtenbaum}.

\subsubsection*{Notation and conventions}
By a variety, we mean an integral separated scheme of finite type over a field $k$.
In particular, all our varieties are irreducible.
Curves and surfaces are smooth of dimensions one and two over $k$, respectively.
We always assume that the ground field \( k \) is of characteristic not two.

When \(X\) is a variety over \(\C\), we write \(\KO^*(X)\) and \(\K^*(X)\) for the real and complex topological K-groups of the set of complex points \(X(\C)\) equipped with the analytic topology.
The quotients \mbox{\(\KO^{2i}(X)/\K^0(X)\)} are often denoted as \(\tKOK{2i}{(X)}\), the quotients \mbox{\(\KO^{2i}(X;\Z/2)/\K^0(X;\Z/2)\)} as \(\tKOK{2i}{(X;\Z/2)}\), and similarly for related quotients.

A tilde on a K-, Witt- or Grothendieck-Witt group of a variety \(X\) indicates that we are considering the
corresponding reduced group, \ie the kernel of the pullback along a geometric point \(\bar x\to X\).
In particular, \(\rK_0(X)\), \(\rGW^0(X)\) and \(\rW^0(X)\) denote the kernels of the rank homomorphisms
(\cf Example~\ref{eg:W-of-geometric-point} below). A tilde on any of the groups \(\K^i(X)\), \(\KO^i(X)\) or \(\tKOK{2i}{(X)}\) of a path-connected  topological space \(X\) likewise denotes the kernel of the pullback along the inclusion of a point.

\encouragepagebreak{\sectionair}
\section{Background on Witt groups}
\subsection{Symmetric bundles and complexes}
\paragraph{Symmetric bundles.}
A symmetric bundle \( (\vb{E},\eps)\) over a variety \( X\) is a vector bundle\footnote{%
By convention, we identify a vector bundle with its sheaf of sections. Thus, the common notation \( \OO \) for the sheaf of regular functions on \( X\) will be used to denote the trivial line bundle over \( X \).
}
\( \vb{E}\) over \( X\) equipped with a non-degenerate symmetric bilinear form
\(
\eps\colon{\vb{E}\otimes\vb{E}\rightarrow\lb{O}}
\).
Alternatively, we may view \( \eps\) as an isomorphism from \( \vb{E}\) to its dual bundle \( \vb{E}^{\dual}\). In this case, the symmetry of \( \eps\) is
encoded by the fact that it agrees with its dual \( \eps^{\dual}\)
under the canonical identification of \( \vb{E}\) with \(
(\vb{E}^{\dual})^{\dual}\).

Two symmetric bundles \( (\vb{E},\eps)\) and \( (\vb{F},\varphi)\) are isometric if there is an isomorphism of vector bundles \( i\colon{\vb{E}\rightarrow\vb{F}}\) compatible with the symmetries, \ie such that \( i^{\dual}\varphi i = \eps\). The orthogonal sum of two symmetric bundles has the obvious definition \( {(\vb{E},\eps) \perp (\vb{F},\varphi)}:={(\vb{E}\oplus\vb{F},\eps\oplus\varphi)}\).

\begin{example}[Symmetric line bundles]\label{eg:Symmetric-line-bundles}
Let \( \Pic(X)[2]\) be the subgroup of line bundles of order~\mbox{\( \leq  2\)} in the Picard group \( \Pic(X)\). Any line bundle \( \lb{L}\in\Pic(X)[2]\) defines a symmetric bundle over \( X\). When \( X\) is a projective variety over an algebraically closed field, all symmetric line bundles arise in this way.
In general, the set of isometry classes of symmetric line bundles over \( X\) is described by \( H^1_{\et}(X;\Z/2)\), the first {\'e}tale cohomology group of \( X\) with coefficients in \( \Z/2=\mathrm{O}(1)\). The Kummer sequence exhibits this group as an extension of \( \Pic(X)[2]\):
\begin{equation*}
 0\rightarrow \frac{\OO^*(X)}{\OO^*(X)^2}\rightarrow H^1_{\et}(X;\Z/2) \rightarrow \Pic(X)[2] \rightarrow 0
\end{equation*}
The additional contribution comes from symmetric line bundles of the form \( (\lb{O},\varphi)\), where  \( \varphi\) is some invertible regular function which has no globally defined square root. For example, the trivial line bundle over the punctured disk \( \A^1-0\) carries a non-trivial symmetric form given by multiplication with the standard coordinate function.
\end{example}

More generally, we may consider a variety $X$ together with a fixed line bundle $\lb{L}\in\Pic(X)$. Then a symmetric bundle over $(X,\lb{L})$ is a vector bundle $\vb{E}$ together with a non-degenerate symmetric bilinear form $\eps\colon{\vb{E}\otimes\vb{E}\rightarrow \vb{L}}$. In other words, $\vb{E}$ is symmetric with respect to the ``twisted duality'' on vector bundles defined by $\vb{E}^{\dual\!\lb{L}}:={\mathcal{H}om}(\vb{E},\lb{L})$.

\begin{example}[Hyperbolic bundles]\label{eg:Hyperbolic-bundles}
Any vector bundle $\vb{E}$ gives rise to a symmetric bundle $H_{\lb{L}}(\vb{E}):=(\vb{E}\oplus\vb{E}^{\dual\!\lb{L}},\mm{0&1\\1&0})$ over $(X,\lb{L})$, the hyperbolic bundle associated with $\vb{E}$.
\end{example}

Hyperbolic bundles are the simplest members of a wider class of so-called metabolic bundles: symmetric bundles $(\vb{M},\mu)$ which contain a subbundle
\(
j\colon{\vb{N}\rightarrow\vb{M}}
\)
of half their own rank on which $\mu$ vanishes. In other words, $(\vb{M},\mu)$ is metabolic if it fits into a short exact sequence of the form
\begin{equation}\label{seq:metabolic}
 0\rightarrow \vb{N} \xrightarrow{\;\;j\;\;} \vb{M} \xrightarrow{j^{\dual\!\lb{L}}\mu} \vb{N^{\dual\!\lb{L}}} \rightarrow 0
\end{equation}
The subbundle $\vb{N}$ is then called a Lagrangian of $\vb{M}$. If the sequence splits then $(\vb{M},\mu)$ is isometric to $H_{\lb{L}}(\lb{N})$, at least in any characteristic other than two (\cf \cite{Balmer:Handbook}*{Example~1.1.21}).

\paragraph{Symmetric complexes.} A symmetric complex over $(X,\lb{L})$ is a bounded complex of vector bundles
\begin{equation*}
  \cc{\vb{E}}\colon{}\quad \cdots \rightarrow \vb{E}_2\rightarrow\vb{E}_{1}\rightarrow\vb{E}_{0}\rightarrow \vb{E}_{-1}\rightarrow\cdots
\end{equation*}
together with a symmetric quasi-isomorphism from $\cc{\vb{E}}$ to its
dual complex
\begin{equation*}
  \cc{\vb{E}}^{\dual\!\lb{L}}\colon{}\quad \cdots \rightarrow \vb{E}_{-2}^{\dual\!\lb{L}}\rightarrow\vb{E}_{-1}^{\dual\!\lb{L}}\rightarrow\vb{E}_{0}^{\dual\!\lb{L}}\rightarrow \vb{E}_{1}^{\dual\!\lb{L}}\rightarrow\cdots
\end{equation*}
More generally, an $i$-symmetric complex is a complex $\cc{\vb{E}}$ together with a symmetric quasi-isomorphism
\begin{equation*}
  \eps\colon{\cc{\vb{E}}\overset{\simeq}\longrightarrow\cc{\vb{E}}^{\dual\!\lb{L}}[i]}
\end{equation*}
where $[i]$ denotes an $i$-fold shift to the left. Here, symmetry
means that $\eps^{\dual\!\lb{L}}[i]$ agrees with $\eps$ under the
canonical identification of $\cc{\vb{E}}$ with $(\cc{\vb{E}}^{\dual\!\lb{L}}[i])^{\dual\!\lb{L}}[i]$.

\begin{example}[A $1$-symmetric complex over $\P^1$]\label{eg:generatorW1P1}
Consider the complex of vector bundles \( \OO(-1)\xrightarrow{\cdot x} \OO\) over the projective line \( \P^1\) with coordinates \( [x:y]\). Place \( \OO\) in degree zero. Multiplication by \( y\) induces a symmetric quasi-isomorphism with the dual complex shifted one to the left, so that we obtain a \( 1\)-symmetric complex
\begin{align}\label{eq:generatorW1P1}
        \Psi_0 &:=\left(\vcenter{\xymatrix{
        {\OO(-1)}\ar[r]^{\cdot{x}}\ar^{\cdot{y}}[d]     &{\OO}\ar[d]^{\cdot(-y)}\\
        {\OO}\ar[r]^{\cdot(-x)}                                         &{\OO(1)}
    }}\right)
\end{align}
Analogues of this $1$-symmetric complex over projective curves of higher genus will be described in Remark~\ref{rem:generatorW1C}.
\end{example}

The notions of hyper- and metabolic bundles generalize to complexes in a straight-forward way. In particular, for any bounded complex of vector bundles $\cc{\vb{E}}$ over $X$ and any line bundle $\lb{L}$ over $X$, we have associated $i$-symmetric hyperbolic complexes $H^i_{\lb{L}}(\cc{\vb{E}})$ over $(X,\lb{L})$.

\encouragepagebreak{\subsectionair}
\subsection{\GrothendieckWitt and Witt groups}
Recall that the algebraic K-group $\K_0(X)$ of a variety $X$ is the free abelian group on isomorphism classes of vector bundles over $X$ modulo the relation $\vb{F}\sim\vb{E}+\vb{G}$ for any short exact sequence of vector bundles $0\rightarrow\vb{E}\rightarrow\vb{F}\rightarrow\vb{G}\rightarrow 0$. The \GrothendieckWitt group of $X$ is defined similarly, in terms of symmetric bundles. More precisely, the \GrothendieckWitt group $\GW^0(X,\lb{L})$ of a variety $X$ with a fixed line bundle $\lb{L}$ is the free abelian group on isometry classes of symmetric  bundles over $(X,\lb{L})$ modulo the following two relations:
\begin{align*}
   (\vb{E},\eps)\perp(\vb{G},\gamma) &\sim (\vb{E},\eps) + (\vb{G},\gamma)
     &&\text{for arbitary symmetric bundles $(\vb{E},\eps)$ and $(\vb{G},\gamma)$}\\
   (M,\mu)&\sim H_{\lb{L}}(\vb{N})
     &&\text{for any metabolic bundle $(M,\mu)$ with Lagrangian $\vb{N}$}
\end{align*}
The Witt group $\W^0(X,\lb{L})$ is the quotient of the
\GrothendieckWitt group by the subgroup generated by metabolic
bundles, \ie it is the quotient of $\GW^0(X,\lb{L})$ by the image of
$\K_0(X)$ under the hyperbolic map $H_{\lb{L}}$. \cite{Knebusch:varieties}*{\S~4}

Recall that we always assume our ground field to be of characteristic not two.  Under this assumption, the groups $\GW^0(X,\lb{L})$ and $\W^0(X,\lb{L})$ may alternatively be defined in terms of the bounded derived category of vector bundles over $X$, just as in the case of $\K_0(X)$. In particular, elements of $\GW^0(X,\lb{L})$ and $\W^0(X,\lb{L})$ may be represented by $0$-symmetric complexes. More generally, by using $i$-symmetric complexes, one may define shifted groups $\GW^i(X,\lb{L})$ and $\W^i(X,\lb{L})$. The theory is best set up by introducing a general notion of a triangulated category with duality and defining \GWorW groups in this context, as pioneered by Balmer and Walter in \cites{Balmer:TWGI,Balmer:TWGII,Walter:TGW}.

Thus, for any variety $X$ with a fixed line bundle $\lb{L}$ and
for any integer $i$ we have a \GrothendieckWitt group $\GW^i(X,\lb{L})$ equipped with hyperbolic and forgetful maps
\begin{align*}
  H^i_{\lb{L}}\colon{}\K_0(X)\rightarrow\GW^i(X,\lb{L})\\
  F\colon{\GW^i(X,\lb{L})\rightarrow\K_0(X)}
\end{align*}
The Witt group $\W^i(X,\lb{L})$ is defined as the quotient of $\GW^i(X,\lb{L})$ by the image of $\K_0(X)$ under $H^i_{\lb{L}}$. Moreover, we have exact sequences of the following form, which we will refer to as the Karoubi sequences:
\begin{equation}\label{seq:Karoubi}
\GW^{i-1}(X,\lb{L})\xrightarrow{F}\K_0(X)\xrightarrow{H^i_{\lb{L}}}\GW^i(X,\lb{L})\rightarrow\W^i(X,\lb{L})\rightarrow 0
\end{equation}
It turns out that the \GWorW groups only depend on the class of $\lb{L}$ in $\Pic(X)/2$, and by convention we will usually drop $\lb{L}$ from the notation when it is trivial in $\Pic(X)/2$. Moreover, the groups are $4$-periodic in $i$. Of course, for $i\equiv 0$, we recover the \GrothendieckWitt groups defined in terms of symmetric bundles. Similarly, the groups $\GW^2(X,\lb{L})$ may be defined in terms of \emph{anti}-symmetric bundles. The groups $\GW^1(X,\lb{L})$ and $\GW^3(X,\lb{L})$, on the other hand, are more intricate.

Both \GrothendieckWitt and Witt groups are functorial:  a morphism of varieties $f\colon{X\rightarrow Y}$ induces pullback maps
\[f^*\colon{\W^i(Y,f^*(\lb{L}))\rightarrow\W^i(X,\lb{L})}\]
 What is more, shifted Witt groups constitute a cohomology theory on smooth varieties: for any smooth closed subvariety $Z\subset X$ we have localization sequences relating the Witt groups of $X$ to those of $Z$ and its complement $U:=X-Z$. That is, for any line bundle $\lb{L}$ over $X$ we have long exact sequences
\begin{align}\label{seq:localization}
\notag &\GW^{i-c}(Z,\lb{L}')\rightarrow  \GW^i(X,\vb{L}) \rightarrow \GW^i(U,\vb{L}|_U)\\
       &\quad\rightarrow \W^{i+1-c}(Z,\lb{L}') \rightarrow \W^{i+1}(X,\vb{L}) \rightarrow \W^{i+1}(U,\vb{L}|_U)\\
\notag &\quad\quad\rightarrow \W^{i+2-c}(Z,\lb{L}') \rightarrow \W^{i+2}(X,\vb{L})\rightarrow \W^{i+2}(U\vb{L}|_U) \rightarrow \cdots
\end{align}
Here, $c$ is the codimension and $\vb{N}$ is the normal bundle of $Z$ in $X$, and $\lb{L}'=\det(\vb{N})\otimes\lb{L}|_Z$.

\begin{example}[Witt groups of a geometric point]\label{eg:W-of-geometric-point}
Let $\bar x=\mathrm{Spec}(k)$, where $k$ is an algebraically closed field of characteristic not two. Then one finds that
\begin{align*}
   &\GW^0(\bar x) = \Z   && \W^0(\bar x) = \Z/2 \\
   &\GW^1(\bar x) = 0    && \W^1(\bar x) = 0    \\
   &\GW^2(\bar x) = \Z   && \W^2(\bar x) = 0    \\
   &\GW^3(\bar x) = \Z/2 && \W^3(\bar x) = 0
\end{align*}
The group $\GW^0(\bar x)$ is generated by the trivial symmetric bundle $(\lb{O},\id)$, whereas $\GW^2(\bar x)$ and $\GW^3(\bar x)$ are generated by the hyperbolic bundles $H^2(\OO)=(\OO\oplus\OO,\mm{0 & 1\\ -1 & 0})$ and $H^3(\OO)$, respectively.
\end{example}
We define the reduced Witt- and \GrothendieckWitt groups of a variety \(X\) to be the kernels of the pullbacks along a geometric point \(\bar x \to X\).  On \(\GW^0\) and \(\W^0\), these pullback maps may be identified with the rank homomorphism and the rank homomorphism modulo two, respectively:
\begin{align*}
  \rGW^0(X) &= \ker(\rank\colon \GW^0(X)\twoheadrightarrow \Z)\\
  \rW^0(X) &= \ker(\bar\rank\colon \W^0(X)\twoheadrightarrow \Z/2)
\end{align*}

\begin{example}[Witt groups of $\P^1$ \cites{Arason,Walter:PB}]
For the projective line $\P^1$ over an algebraically closed field of characteristic not two, one finds that
\begin{align*}
  &\W^0(\P^1)=\Z/2\\
  &\W^1(\P^1)=\Z/2\\
  &\W^2(\P^1)=0\\
  &\W^3(\P^1)=0
\end{align*}
The group $\W^0(\P^1)$ is again generated by the trivial symmetric bundle, whereas $\W^1(\P^1)$ is generated by the $1$-symmetric complex $\Psi_0$ given in Example~\ref{eg:generatorW1P1}. The twisted Witt groups $\W^i(\P^1,\OO(1))$ vanish.
\end{example}

In general, when working over an algebraically closed ground field, the Witt groups are always \( 2\)-torsion: \( 2\Psi=0\) for any \( \Psi\in\W^i(X;\lb{L})\). This is a consequence of the following well-known lemma.
\begin{lemma}\label{lem:HF}
Let \( X\) be a variety over an algebraically closed field \( k\) of characteristic not two. Then, for any \( \Psi\in\GW^i(X;\lb{L})\), we have \( H^i_{\lb{L}}(F(\Psi)) \cong 2\Psi\).
\end{lemma}
\begin{proof}
We may assume that \( \Psi\) is the class of some \( i\)-symmetric complex \( (\cc{\vb E},\eps)\). Using the assumption that \( \mathrm{char}(k)\neq 2\), we define an isometry between the hyperbolic complex \( H^i_{\lb{L}}(\cc{\vb{E}})\) and the direct sum \( (\cc{\vb{E}},\eps)\oplus(\cc{\vb{E}},-\eps)\) by \( \frac{1}{2}\mm{1&\eps^{-1}\\1&-\eps^{-1}}\). Since \( k\) contains a square root of \( -1\), the symmetric complex \( (\cc{\vb{E}},-\eps)\) is isometric to \( (\cc{\vb{E}},\eps)\).
\end{proof}

\encouragepagebreak{\subsectionair}
\subsection{Comparison with KO-theory}
Now suppose $X$ is a smooth variety over $\C$. Then there is an
obvious comparison map from the algebraic K-group $\K_0(X)$ to the
topological K-group $\K^0(X)$ of $X$ equipped with the analytic
topology, sending algebraic vector bundles over $X$ to the underlying
continuous complex vector bundles. In \cite{Me:WCCV}, we explain how
certain mild hypotheses concerning the representability of Witt
groups, or more generally of hermitian K-theory, in the stable
$\A^1$-homotopy category may be used to obtain a convenient definition of
analogous maps to the topological KO-groups of $X$:
\begin{align*}
   \gw^i\colon{}&\GW^i(X,\lb{L})\rightarrow \KO^{2i}(X,\lb{L})\\
   \w^i\colon{}&\W^i(X,\lb{L})\rightarrow \KO^{2i-1}(X,\lb{L})
\end{align*}
(For non-trivial line bundles $\lb{L}$ over $X$, we put
\( \KO^i(X,\lb{L}):=\rKO^{i+2}(\Thom_X(\lb{L})) \) -- \cf
\cite{Me:WCCV}*{Section~2.1}.)
``Convenient'' here means that many useful properties of these maps
follow directly from their definition. In particular, we have the
following two statements:
\begin{itemize}
\item The maps $\w^i$ and $\gw^i$ respect the Karoubi sequences \eqref{seq:Karoubi} in the sense that we have commutative diagrams
\begin{smalldiagram}
\vcenter{
\xymatrix{
  & {\GW^{i-1}\!{X}} \ar[r]\ar[d]
  & {\K_0\!{X}}      \ar[r]\ar[d]
  & {\GW^i\!{X}}     \ar[r]\ar[d]
  & {\W^i\!{X}}      \ar[r]\ar[d]
  & {0}              \ar[d]
  &
  \\
  {\cdots} \ar[r]
  & {\KO^{2i-2}\!{X}}\ar[r]
  & {\K^0\!{X}}      \ar[r]
  & {\KO^{2i}\!{X}}  \ar[r]
  & {\KO^{2i-1}\!{X}}\ar[r]
  & {\K^1\!{X}} \ar[r]
  & {\cdots}
  }}
\end{smalldiagram}%
Here, the lower sequence is part of the long exact sequence known as the Bott sequence in topology. It follows in particular that the maps $\w^i$ factor through the quotients $\tKOK{2i}{(X)}$, so that we have induced maps
\begin{align*}
  \w^i\colon{}&\W^i(X)\rightarrow \tKOK{2i}{(X)}
\end{align*}
\item The maps respect localization sequences \eqref{seq:localization}, in the sense that we have commutative ladder diagrams:
\begin{smalldiagram}
\vcenter{\xymatrix@C=5pt{
  {\GW^{i-c}(Z,\lb{L}')} \ar[r]\ar[d]^{\gw^{i-c}}
& {\GW^i(X,\vb{L})}      \ar[r]\ar[d]^{\gw^i}
& {\GW^i(U,\vb{L}|_U)}   \ar[r]\ar[d]^{\gw^i}
& {\W^{i-c+1}(Z,\lb{L}')}\ar[r]\ar[d]^{\w^{i-c+1}}
& {\W^{i+1}(X,\vb{L})}   \ar[r]\ar[d]^{\w^i}
& {\W^{i+1}(U,\vb{L}|_U)}      \ar[d]^{\w^i}\\
  {\KO^{2(i-c)}(Z,\lb{L}')}   \ar[r]
& {\KO^{2i}(X,\vb{L})}       \ar[r]
& {\KO^{2i}(U,\vb{L}|_U)}    \ar[r]
& {\KO^{2(i-c)+1}(Z,\lb{L}')} \ar[r]
& {\KO^{2i+1}(X,\vb{L})}     \ar[r]
& {\KO^{2i+1}(U,\vb{L}|_U)}
}}
\end{smalldiagram}%
Moreover, these ladders may be extended to the left and the right in a commutative way. We refer to \cite{Me:WCCV}*{Section~2} for details.
\end{itemize}
\begin{remark}\label{rem:disclaimer}
The precise form of our hypotheses is given in
\cite{Me:WCCV}*{Standing Assumptions 1.9}. We refrain from reproducing
the details given there and point to \cite{Me:Thesis}*{Chapter~II} for a
more thorough discussion. The cautious reader is advised that the
results in Sections~\ref{sec:CS:comparison:shifted}~ff.\ do depend on the properties of the comparison
maps following from these hypotheses. All calculations in
Sections~\ref{sec:CS:Curves} and \ref{sec:CS:Surfaces}, however, as well as the
comparison result for the classical Witt group,
Proposition~\ref{prop:comparison_0}, are valid without any such disclaimer.
\end{remark}

\encouragepagebreak{\subsectionair}
\subsection{Stiefel-Whitney classes}\label{sec:CS:SW}
\newcommand{\RP}{\R\P}
As we will see, all elements in the Witt group \( \W^0(X) \) of a
curve or surface can be detected by the rank homomorphism and the
first two Stiefel-Whitney classes. We therefore include a brief account of the general theory of these characteristic classes.

\subsubsection{Stiefel-Whitney classes over varieties}
\label{sec:CS:SW:varieties}
The following construction of Stiefel-Whitney classes over varieties due to Delzant
\cite{Delzant} and Laborde \cite{Laborde} is detailed in
\cite{EKV}*{\S~5}. It works more generally for any scheme \(X\) over \(\Z[\frac{1}{2}]\).

The first Stiefel-Whitney class \( w_1 \) of a symmetric line bundle over \( X \) is defined by the correspondence of isometry classes of such bundles with elements in \( H^1_{\et}(X;\Z/2) \) (see Example~\ref{eg:Symmetric-line-bundles}). For an arbitrary symmetric bundle \( (\vb{E},\eps) \), one considers the scheme \( \P_{\nondeg}(\vb{E},\eps) \) given by the complement of the quadric in \( \P(\vb{E}) \) defined by \( \eps \). The restriction of the universal line bundle \( \OO(-1) \) on \( \P(\vb{E}) \) to \( \P_{\nondeg}(\vb{E},\eps) \) carries a canonical symmetric form. Let \( w \) be its first Stiefel-Whitney class in \( H^1_{\et}(\P_{\nondeg}(\vb{E},\eps);\Z/2) \).
Then the cohomology of \( \P_{\nondeg}(\vb{E},\eps) \) can be decomposed as
\begin{equation}\label{eq:SW_cohomology-of-RP}
   H^*_{\et}(\P_{\nondeg}(\vb{E},\eps);\Z/2)=\bigoplus_{i=0}^{r-1} p^*H^*_{\et}(X;\Z/2)\cdot w^i
\end{equation}
where \( p \) is the projection of \( \P_{\nondeg}(\vb{E},\eps) \) onto \( X \). In particular, \( w^r \) is a linear combination of smaller powers of \( w \), so that for certain coefficients
\begin{equation*}
   w_i(\vb{E},\eps)\in H^i(X;\Z/2)
\end{equation*}
we have
\begin{equation*}
   w^r + p^*w_1(\vb{E},\eps)\cdot w^{r-1} + p^*w_2(\vb{E},\eps)\cdot w^{r-2} + \dots + p^*w_r(\vb{E},\eps) = 0
\end{equation*}
These coefficients \( w_i(\vb{E},\eps) \) are defined to be the Stiefel-Whitney classes of \( (\vb{E},\eps) \). They are characterized by the following axiomatic description \cite{EKV}*{\S~1}:
\\[0.5\baselineskip]
\newcommand*\mylisting[2]{%
   \noindent
 \parbox[t]{\widthof{\textbf{Normalization. }}}{\raggedright \textbf{#1.}}%
 \parbox[t]{\linewidth-\widthof{\textbf{Normalization. }}}{#2}
   \\[0.5\baselineskip]
}
\mylisting{Normalization}{%
   The first Stiefel-Whitney class of a symmetric line bundle \( w_1 \) is as defined above.%
   }%
\mylisting{Boundedness}{%
   \( w_i(\vb{E},\eps)=0\; \) for all \( \;i>\rank(\vb{E}) \)%
   }%
\mylisting{Naturality}{%
   For any morphism \( f\colon{Y\rightarrow X} \), we have
 \( f^*w_i(\vb{E},\eps)=w_i(f^*(\vb{E},\eps)) \).%
   }%
\mylisting{Whitney sum formula}{%
  The total Stiefel-Whitney class \( w_t:=1+\sum_{i\geq 1} w_i(\vb{E},\eps)t^i \) satisfies
 \begin{equation*}
  w_t((\vb{E},\eps)\perp(\vb{F},\varphi)) = w_t(\vb{E},\eps)\cdot w_t(\vb{F},\varphi)
 \end{equation*}%
}
\shortendisplayskip

The Stiefel-Whitney classes of a metabolic bundle only depend on the Chern classes of its Lagrangian. More precisely, Proposition~5.5 in \cite{EKV} gives the following formula for a metabolic bundle \( (\vb{M},\mu) \) with Lagrangian \( \vb{L} \):
\begin{align}
 w_t(\vb{M},\mu) &= \sum_{j=0}^{\rank(\vb L)}(1+(-1)t)^{\rank(\vb L)-j}c_j(\vb{L})t^{2j} \label{eq:dt}
\intertext{%
For example, for the first two Stiefel-Whitney classes we have
}
 w_1(\vb{M},\mu) &= \rank(\vb{L})(-1)  \label{eq:dt:1}\\
 w_2(\vb{M},\mu) &= \binom{\rank(\vb{L})}{2}(-1,-1) + c_1(\vb{L}) \label{eq:dt:2}
\end{align}
It follows that \( w_t \) descends to a well-defined homomorphism from the \GrothendieckWitt group of \( X \) to the multiplicative group of invertible elements in \( \bigoplus_i H^i_{\et}(X;\Z/2)t^i \):
\begin{align*}
 w_t&\colon{\GW^0(X)\longrightarrow \left(\bigoplus_i H^i_{\et}(X;\Z/2)t^i\right)^{\times}}
\intertext{%
In particular, the individual classes descend to well-defined maps
}%
 w_i&\colon\GW^0(X)\rightarrow H^i_{\et}(X;\Z/2)
\end{align*}

In general, none of the individual Stiefel-Whitney classes apart from \( w_1 \) define homomorphisms on \( \GW^0(X) \). It does follow from the Whitney sum formula, however, that \( w_2 \) restricts to a homomorphism on the kernel of \( w_1 \), and in general \( w_i \) restricts to a homomorphism on the kernel of (the restriction of) \( w_{i-1} \).

It is not generally true either that the Stiefel-Whitney classes factor through the Witt group of \( X \): the right-hand side of \eqref{eq:dt} may be non-zero. However, we may deduce from \eqref{eq:dt:1} that the restriction of the first Stiefel-Whitney class to the reduced group \( \rGW^0(X) \) factors through \( \rW^0(X) \), yielding a map
\begin{align*}
 \bar{w}_1&\colon{\rW^0(X)}\longrightarrow H^1_{\et}(X;\Z/2)
\intertext{%
We use a different notation to emphasize that the values of \( w_1 \)
and \( \bar{w}_1 \) on a given symmetric bundle  may differ. That is,
if \( (\vb{E},\eps) \) is a symmetric bundle of even rank defining an
element \( \Psi \) in \( \rW^0(X) \), then in general \(
\bar{w}_1(\Psi) \neq w_1(\vb E,\eps) \). Rather, \( \bar{w}_1 \) needs
to be computed on a lift of \( (\vb E,\eps) \) to \( \rGW^0(X)
\). Equation \eqref{eq:dt:2} implies that \( w_2 \) also  induces a well-defined map
}
 \bar{w}_2&\colon{\rW^0(X)}\longrightarrow \quotient{H^2_{\et}(X;\Z/2)}{\Pic(X)}
\end{align*}
As before, \( \bar{w}_1 \) is a surjective homomorphism, while \( \bar{w}_2 \) restricts to a homomorphism on \( \ker(\bar{w}_1) \).

\subsubsection{Stiefel-Whitney classes over fields}
\label{sec:CS:SW:fields}
When \( X \) is a field \( F \) (of characteristic not two), the Stiefel-Whitney classes factor through Milnor's K-groups modulo two, which are commonly denoted \( k^M_i(F) \). We will denote the classes with values in \( k^M_i(F) \) by \( w_i^M \):
\begin{equation*}
  \xymatrix{
  {\GW^0(F)}\ar[rd]_{w_i^M}\ar[rr]^{w_i} & & H^i_{\et}(F;\Z/2)\\
   & k_i^M(F)\ar[ru]_{\alpha}
  }
\end{equation*}
Both the groups \( k_i^M(F) \) and the classes \( w_i^M \) were constructed in \cite{Milnor:Conjectures}. In the same paper, Milnor asked whether the map \( \alpha \) appearing in the factorization was an isomorphism, a question that later became known as one of the Milnor conjectures. For \( i\leq 2 \), which will be the range mainly relevant for us, an affirmative answer was given by Merkurjev \cite{Merkurjev:MC}. A general affirmation of the conjecture was found more recently by Voevodsky \cite{Voevodsky:MC-with-2}.

A second conjecture of Milnor, also to be found in
\cite{Milnor:Conjectures}, concerned the relation of \( \W^0(F) \) to
\( k_i^M(F) \). To state it, we introduce the fundamental
filtration. If we view \( \W^0(F) \) as a ring, then the kernel of the
rank homomorphism \( \rank\colon\W^0(F)\rightarrow\Z/2 \) becomes an ideal inside \( \W^0(F) \), which is traditionally written as \( I(F) \). The powers of this ideal yield a filtration
  \[ \W^0(F)\supset I(F)\supset I^2(F)\supset I^3(F)\supset \dots \]
on the Witt ring of \( F \), known as the fundamental filtration. Milnor conjectured that the associated graded ring was isomorphic to \( k_*^M(F):=\oplus_i k_i^M(F) \). As a first step towards a proof, he constructed maps \( k_i^M(F)\rightarrow I^i(F)/I^{i+1}(F) \) in one direction. Moreover, in degrees one and two, Milnor could show that these are isomorphisms, with explicit inverses induced by the Stiefel-Whitney classes \( w_1^M \) and \( w_2^M \). In combination with the isomorphisms \(  \alpha  \) above, one obtains the following identifications:
\begin{align}
 \notag \rank     &\colon \;\quotient{\W^0(F)}{I(F)}\;\; \overset{\cong}\longrightarrow H^0_{\et}(F;\Z/2)\\
 \notag \bar{w}_1 &\colon \;\;\;\;\quotient{I(F)}{I^2(F)}\; \overset{\cong}\longrightarrow H^1_{\et}(F;\Z/2)\\
 \notag \bar{w}_2 &\colon \;\;\quotient{I^2(F)}{I^3(F)}\; \overset{\cong}\longrightarrow H^2_{\et}(F;\Z/2)
\intertext{%
\cite{Milnor:Conjectures}*{\S~4}. Today, these isomorphisms are commonly denoted \( e^0 \), \( e^1 \), \( e^2 \). It was clear from the outset, however, that the higher isomorphisms}
 e^i       &\colon \quotient{I^i(F)}{I^{i+1}(F)} \overset{\cong}\longrightarrow H^i_{\et}(F;\Z/2)\label{eq:Milnors-isos}
\end{align}
conjectured by Milnor could not be induced by higher Stiefel-Whitney
classes. Their existence was ultimately proved in
\cite{OVV:MilnorConjecture}. Unlike in the case of Stiefel-Whitney
classes, it does not seem to be clear how these isomorphisms may be
generalized to varieties (\cf \cite{Auel:MilnorRemarks}).

\subsubsection{Stiefel-Whitney classes over complex varieties}
\label{sec:CS:SW:complex}
Over a complex variety, the {\'e}tale Stiefel-Whitney classes of symmetric vector bundles are compatible with the Stiefel-Whitney classes of real vector bundles used in topology:

Suppose first that \( Y \) is an arbitrary CW complex. If we follow the construction of Stiefel-Whitney classes described above, with singular cohomology in place of {\'e}tale cohomology, we obtain classes \(  w_i(\vb{E},\eps) \) in \(  H^i(Y; \Z/2)  \) for every complex symmetric bundle \(  (\vb E, \eps)  \) over \(  Y  \).  On the other hand, given a real vector bundle \(  \vb F  \) over \(  Y  \), we have the usual Stiefel-Whitney classes \(  w_i(\vb F) \). We claim that these classes are compatible in the following sense. Recall that we have a one-to-one correspondence
\begin{equation*}
 \Re\colon \left(\txt{isometry classes of \\ complex symmetric \\
     bundles over \( Y \)}\right)
   \rightarrow \left(\txt{isomorphism classes \\ of real vector
      bundles \\ over \( Y \)}\right)
\end{equation*}
Explicitely, \( \Re \) sends a complex symmetric vector bundle \( (\vb
E, \eps) \) to the unique real subbundle \( \Re(\vb E, \eps) \) of \(
\vb E \) on which \( \eps \) restricts to a real positive definite
form (\cf \cite{Me:WCCV}*{Lemma~1.3 and its corollary}). It follows
from the construction that the two topological versions of
Stiefel-Whitney classes discussed agree under \( \Re \):
\begin{lemma}[\cite{Me:Thesis}*{Lemma~III.1.1}]
For any complex symmetric vector bundle \(  (\vb E, \eps)  \) over a CW complex \(  Y  \), the classes
\(  w_i(\vb{E},\eps) \) and  \(  w_i(\Re(\vb E,\eps)) \) in \(  H^i(Y;
\Z/2)  \) agree. \qed
\end{lemma}

Now let \( X \) be a complex variety. Then the lemma implies that the
{\'e}tale and the topological Stiefel-Whitney classes are compatible
in the sense that we have a commutative square as follows:
\begin{equation*}
\xymatrix{
    {\left(\txt{symmetric vector bundles \\ over $X$}\right)}\ar[d]\ar[r]^-{w_i}
    & {H^i_{\et}(X;\Z/2)}\ar[d]^{\cong}\\
    {\left(\txt{real vector bundles\\over $X(\C)$}\right)}   \ar[r]^-{w_i}
    & {H^i(X(\C);\Z/2)}
   }
\end{equation*}

Moreover, if we specialize equation~\eqref{eq:dt} for the Stiefel-Whitney classes of a metabolic bundle \( (\vb{M},\mu) \) with Lagrangian \( \vb{L} \) to the case of a complex variety, then since \( -1 \) is a square in \( \C \) we find that
\begin{equation*}
   w_{2i}(\vb{M},\mu) = c_i(\vb{L}) \in H_{\et}^{2i}(X;\Z/2)
\end{equation*}
and all odd Stiefel-Whitney classes of \( (\vb{M},\mu) \) vanish. This corresponds to the well-known fact in topology that the even Stiefel-Whitney classes of a complex vector bundle agree with its Chern classes modulo two, whereas its odd Stiefel-Whitney classes are zero \cite{MilnorStasheff}*{Problem~14.B}.
It follows in particular that the odd Stiefel-Whitney classes factor through the (reduced) Witt group of \( X \), while the even classes induce maps
\begin{equation*}
  \bar{w}_{2i}\colon \rW^0(X)\longrightarrow \frac{H^{2i}_{\et}(X;\Z/2)}{\im(c_i)}
\end{equation*}
We summarize the situation as follows.

\begin{proposition}\label{prop:Stiefel-Whitney-classes}
Let \( X \) be a complex variety. The Stiefel-Whitney classes factor through the reduced \GrothendieckWitt and KO-group of \( X \) to yield commutative diagrams
\begin{equation*}
   \xymatrix{
   {\rGW^0(X)}\ar[d]\ar[r]^-{w_i} & {H^i_{\et}(X;\Z/2)}\ar[d]^{\cong}\\
   {\rKO^0(X)}\ar[r]^-{w_i}   & {H^i(X;\Z/2)}
   }
\end{equation*}
Moreover, for all odd \( i \) we have induced maps
\begin{equation*}
   \xymatrix{
   {\rW^0(X)}\ar[d]\ar[r]^-{\bar{w}_i} & {H^i_{\et}(X;\Z/2)}\ar[d]^{\cong}\\
   {\trKOK{0}{(X)}}\ar[r]^-{\bar{w}_i} & {H^i(X;\Z/2)}
   }
\end{equation*}
and for even \( i \) we have induced maps
\begin{samepage}
\begin{equation*}
   \xymatrix{
   {\rW^0(X)}\ar[d]\ar[r]^-{\bar{w}_i} & {\dfrac{H^i_{\et}(X;\Z/2)}{\im(c_i)}}\ar@{->>}[d]\\
   {\trKOK{0}{(X)}}\ar[r]^-{\bar{w}_i}    & {\dfrac{H^i(X;\Z/2)}{\im(c_i)}}
   }
\end{equation*}
\shortendisplayskip
\qed
\end{samepage}
\end{proposition}

\encouragepagebreak{\sectionair}
\section{Curves}\label{sec:CS:Curves}
In the next section, we will compute the Witt groups and the groups \(
\tKOK{2i}{(X)} \) for smooth varieties of dimension at most two. Here,
we briefly summarize the results we will obtain in the case of
curves. Moreover, we give a complete description of their
\GrothendieckWitt groups.

\subsection{Grothendieck-Witt groups of curves}
Let \( C \) be a smooth curve over an algebraically closed field \( k \) of characteristic not two, and let \( \Pic(C) \) be its Picard group.

If \( C \) is projective, say of genus \( g \), we may write \( \Pic(C) \) as \( \Z\oplus\Jac(C) \). The free summand \( \Z \) is generated by a line bundle \( \OO(p) \) associated with a point \( p \) on \( C \), while \( \Jac(C) \) denotes (the closed points of) the Jacobian of \( C \), a \( g \)-dimensional abelian variety parametrizing line bundles of degree zero over \( C \). As a group, \( \Jac(C) \) is two-divisible, and \( \Jac(C)[2] \) has rank \( 2g \) (\eg \cite{Milne:LEC}*{Chapter 14}). In particular, when \( C \) is projective, we have
\begin{align*}
 H^1_{\et}(C;\Z/2)&=\Pic(C)[2]=(\Z/2)^{2g}\\
 H^2_{\et}(C;\Z/2)&=\Pic(C)/2=\Z/2
\end{align*}

If \( C \) is not projective, it is affine and may be obtained from a smooth projective curve by removing a finite number of points. Note that when we remove a single point \( p \) from a projective curve \( \bar C \), the Picard group is reduced to \( \Pic(\bar C-p)=\Jac(\bar C) \). It follows that the Picard group of any affine curve is two-divisible. In particular, for any affine curve \( C \) we have
\begin{equation*}
  H^2_{\et}(C;\Z/2)=\Pic(C)/2=0
\end{equation*}

The reduced algebraic K-group of a smooth curve may be identified with its Picard group via the first Chern class, so that we have an isomorphism
\begin{equation}\label{eq:Kcurve}
  \K_0(C)\cong\Z\oplus\Pic(C)
\end{equation}
The following proposition shows that, similarly, the \GrothendieckWitt and Witt groups of \( C \) are completely determined by \( \Pic(C) \) and the group \( H^1_{\et}(C;\Z/2) \) of symmetric line bundles.

\begin{theorem}\label{thm:W_Curve}
Let \( C \) be a smooth curve over an algebraically closed field of characteristic not two. The \GrothendieckWitt and Witt groups of \( C \) are as follows:
\begin{align*}
   &\GW^0(C) = [\Z] \oplus H^1_{\et}(C;\Z/2)\oplus H^2_{\et}(C;\Z/2)  && \W^0(C) = [\Z/2] \oplus H^1_{\et}(C;\Z/2)  \\
   &\GW^1(C) = \Pic(C)                                  && \W^1(C) = H^2_{\et}(C;\Z/2)                    \\
   &\GW^2(C) = [\Z]                                     && \W^2(C) = 0                                \\
   &\GW^3(C) = [\Z/2] \oplus \Pic(C)                    && \W^3(C) = 0
\intertext{Here, the summands in square brackets are the trivial ones coming from a point, \ie those that disappear when passing to reduced groups. In particular, for a projective curve of genus \( g \) we obtain:}
   &\GW^0(C) = [\Z]   \oplus (\Z/2)^{2g+1}             && \W^0(C) = [\Z/2] \oplus (\Z/2)^{2g}  \\
   &\GW^1(C) = \Z     \oplus \Jac(C)                   && \W^1(C) = \Z/2                 \\
   &\GW^2(C) = [\Z]                                    && \W^2(C) = 0                    \\
   &\GW^3(C) = [\Z/2] \oplus \Z \oplus \Jac(C)         && \W^3(C) = 0
\intertext{For affine curves, no non-trivial twists are possible. When \( C \) is projective, the groups twisted by a generator \( \OO(p) \) of \( \Pic(C)/2 \) are as follows:}
   &\GW^0(C,\OO(p)) = \Z \oplus (\Z/2)^{2g}             && \W^0(C,\OO(p)) = (\Z/2)^{2g}         \\
   &\GW^1(C,\OO(p)) = \Z \oplus \Jac(C)                 && \W^1(C,\OO(p)) = 0      \\
   &\GW^2(C,\OO(p)) = \Z                                && \W^2(C,\OO(p)) = 0      \\
   &\GW^3(C,\OO(p)) = \Z \oplus \Jac(C)                 && \W^3(C,\OO(p)) = 0
\end{align*}
\end{theorem}

\begin{remark}[Explicit generators]\label{rem:generatorW1C}
The isomorphism between \( \rW^0(C) \) and \( H^1_{\et}(X;\Z/2) \) is the obvious one, \ie \( \rW^0(C) \) is generated by symmetric line bundles (see Example~\ref{eg:Symmetric-line-bundles}). When \( C \) is projective of genus \( g \), the group \( \W^1(C) \) has a single generator \( \Psi_g \), which for \( g=0 \) we may take to be the \( 1 \)-symmetric complex given in Example~\ref{eg:generatorW1P1}. When \( g\geq 1 \), a generator may be constructed as follows: Pick two distinct points \( p \) and \( q \) on \( C \). Let \( s \) and \( t \) be sections \( \OO\rightarrow\OO(p) \) and \( \OO\rightarrow\OO(q) \) of the associated line bundles that vanish at \( p \) and \( q \), respectively. Choose a line bundle \( \lb{L}\in\Pic(C) \) that squares to \( \OO(p-q) \). Then
\begin{align}\label{eq:generatorW1C}
\Psi_g &:=\left(\vcenter{\xymatrix{{\lb{L}\otimes\OO(-p)}\ar[r]^{s}\ar[d]^{t} & {\lb{L}}\ar[d]^{-t}\\
   {\lb{L}\otimes\OO(q-p)}\ar[r]^{-s}         & {\lb{L}\otimes\OO(q)}\\
   }}\right)
\end{align}
is a \( 1 \)-symmetric complex over \( C \) generating \( \W^1(C) \).
\end{remark}

\proof[Proof of Theorem~\ref{thm:W_Curve}, assuming Theorem~\ref{thm:W_Surface}]
The values of the untwisted Witt groups may be read off directly from Theorem~\ref{thm:W_Surface} below: we simply note that \( c_1\colon \Pic(C)\rightarrow H^2_{\et}(C;\Z/2) \) is surjective and induces an iso\-morphism \( \Pic(C)/2\cong H^2_{\et}(C;\Z/2) \). The twisted groups \( \W^i(C,\OO(p)) \) of a projective curve \( C \) may then be calculated from the localization sequence \eqref{seq:localization} associated with the inclusion of the complement of a point \( p \) of \( C \) into \( C \). Indeed, since the line bundle \( \OO(p) \) is trivial over \( C-p \), this sequences takes the following form:
\begin{align*}
\cdots\rightarrow\W^{i-1}(p)\rightarrow\W^i(C,\OO(p))\rightarrow\W^i(C-p)\rightarrow\W^i(p)\rightarrow\cdots
\end{align*}
We now deduce the values of the \GrothendieckWitt groups, concentrating on the case when \( C \) is projective. The affine case is very similar, except that all extension problems disappear.

\paragraph{Untwisted case.}%
To compute the untwisted \GrothendieckWitt groups of a projective curve \( C\), we use the Karoubi sequences \eqref{seq:Karoubi} and their restrictions to reduced groups:
\begin{align*}
   \rGW^{i-1}(C) \overset{F}{\rightarrow} \rK_0(C) \overset{H^i}{\rightarrow} \rGW^i(C) \twoheadrightarrow \rW^i(C) \rightarrow 0
\end{align*}
Note first that the identification of the K-group of our projective curve \( C \) given in \eqref{eq:Kcurve} may be written explicitly as
\begin{align*}
\K_0(C)               &\overset{\cong}\longleftrightarrow  \Z \oplus \Z \oplus \Jac(C)\\
[\vb{E}]\;            &\mapsto              (\rank\vb{E}, d,\det{\vb{E}}(-d)) \\
[\lb{L}(d)]+(r-1)[\OO] &\mapsfrom            (r,d,\lb{L})
\end{align*}
Here, in the map from left to right, we have written \( d \) for the degree of the determinant line bundle of \( \vb{E} \).
Given this description, the endomorphisms \( F H^i \) of \( \K_0(C) \) may easily be computed directly. For the restrictions to the reduced K-group we obtain:
\begin{align}\label{eq:FHcurve}
  \im(FH^i)=
  \begin{cases}
       0                  & \subset{\rK_0(C)} \quad\text{when \( i \) is even }   \\
       2\Z \oplus \Jac(C) & \subset{\rK_0(C)} \quad\text{when \( i \) is odd}
    \end{cases}
\end{align}
We compute \( \rGW^0(C) \) as in Corollary~\ref{cor:etale-w2_Surface}:  as \( \W^3(C)\) is zero, \( \rK^0(C) \) surjects onto \( \rGW^3(C) \), hence the image of \( \rGW^3(C) \) under \(F\) coincides with the image of \(\rK^0(C)\) under \( FH^3 \), and hence by \eqref{eq:FHcurve} we obtain a short exact sequence
\begin{align*}
   0\rightarrow \Z/2 \longrightarrow \rGW^0(C) \longrightarrow \underset{=\,\Jac(C)[2]}{\rW^0(C)} \rightarrow 0
\end{align*}
Moreover, in this situation the sequence splits:  To define a splitting of the surjection, pick line bundles \( \lb{L} \) of order \( 2 \) defining additive generators \( e_{\lb{L}}:=(\lb{L},\id)-(\OO,\id) \) of \( \rW^0(C) \). We claim that the homomorphism obtained by lifting these to the corresponding elements in \( \GW^0(C) \) is well-defined, \ie that \( 2e_{\lb{L}} \) vanishes in \( \GW^0(C) \). Indeed, by Lemma~\ref{lem:HF} we have \( 2e_{\lb L}=H^0(\lb{L}-\OO) \), and \( H^0 \) vanishes on \( \lb{L}-\lb{O} \) since by \eqref{eq:FHcurve} the latter is contained in \( F(\rGW^3(X)) \).

Next, the image of \( \rGW^0(C) \) in \( \rK_0(C) \) is \( \Jac(C)[2] \). Since the quotient \( \Jac(C)/\Jac(C)[2] \) is isomorphic to \( \Jac(C) \), we obtain a short exact sequence
\begin{align*}
   0 \rightarrow \Z\oplus\Jac(C)\overset{H^1}\longrightarrow \rGW^1(C) \longrightarrow \underset{=\;\Z/2}{\rW^1(C)} \rightarrow 0
\end{align*}
Let \( \Psi_g \) be the \( 1 \)-symmetric complex given by \eqref{eq:generatorW1P1} or \eqref{eq:generatorW1C}.
As an element of \( \rK_0(C)=\Z\oplus\Jac(C) \), its underlying complex is equivalent to \( (1,\OO) \). Thus, by Lemma~\ref{lem:HF} again, we have \( H^1(1,\OO)=H^1(F(\Psi_g))=2\Psi_g \).
It follows that the sequence does not split and that \( \Psi_g \) descends to a generator of \( \rW^1(C) \). Thus, \( \rGW^1(C)=\Z\oplus\Jac(C) \).

Carrying on to compute \( \rGW^2(C) \), we first note that by \eqref{eq:FHcurve} the image of \( \rGW^1(C) \) in \( \rK^0(C) \) contains \( 2\Z\oplus\Jac(C) \). Moreover, \( F(\Psi)=\psi=(1,\OO) \), so in fact \( \rGW^1(C) \) surjects onto \( \rK_0(C) \). It follows that \( \rGW^2(C)=0 \). Finally, since \( \rW^3(C) \) also vanishes, we see that \( H^3 \) gives an isomorphism between \( \rK_0(C) \) and \( \rGW^3(C) \). This completes the computations in the untwisted case.

\paragraph{Twisted case.}
To compute the \GrothendieckWitt groups of \( C \) twisted by a generator \( \OO(p) \) of \( \Pic(C)/2 \) it again seems easiest to compare them with those of the affine curve \( C-p \), over which \( \OO(p) \) is trivial. More specifically, we will compare the respective Karoubi sequences via the following commutative diagram:
\begin{equation}\label{eq:Ctwist_vs_C-p}
\vcenter{\xymatrix@C=9pt{
   {\GW^{i-1}(C-p)}    \ar[r]^-{F}       & {\K_0(C-p)}\ar[r]^{H^i}                & {\GW^i(C-p)}     \ar[r]       & {\W^i(C-p)}     \ar[r]       & {0} \\
   {\GW^{i-1}(C,\OO(p))}\ar[r]^-{F}\ar[u] & {\K_0(C)}  \ar[r]^-{H^i_{\OO(p)}}\ar[u] & {\GW^i(C,\OO(p))} \ar[r]\ar[u] & {\W^i(C,\OO(p))} \ar[r]\ar[u] & {0}
}}
\end{equation}
The restriction map \( \K_0(C)\rightarrow\K_0(C-p) \) is the projection \( \Z\oplus\Z\oplus\Jac(C)\twoheadrightarrow \Z\oplus 0\oplus\Jac(C) \) killing the free summand corresponding to line bundles of non-zero degrees.
The twisted version of \eqref{eq:FHcurve} reads as follows:
\begin{align}\label{eq:FHcurvetwist}
  \im(FH_{\lb{O}(p)}^i)=
  \begin{cases}
        \Z\cdot(2,1)\oplus 0              & \subset \K_0(C) \quad\text{when \( i \) is even } \\
        \;\;0\oplus\;\;\Z \oplus \Jac(C)  & \subset \K_0(C) \quad\text{when \( i \) is odd }
  \end{cases}
\end{align}
Now consider \eqref{eq:Ctwist_vs_C-p} with \( i=0 \). Arguing as in the untwisted case, we may compute the cokernels of the forgetful maps \( F \) to reduce the diagram to a comparison of two short exact sequences:
\begin{equation*}
   \xymatrix{
   {0}\ar[r] & {\Z}\ar[r]^-{\mm{2\\0}} & {\Z\oplus\Jac(C)[2]} \ar[r] & {(\Z/2)\oplus\Jac(C)[2]} \ar[r]       & {0} \\
   {0}\ar[r] & {\Z}\ar[r]\ar@{=}[u]   & {\GW^0(C,\OO(p))}\ar[r]\ar[u] & {\Jac(C)[2]} \ar[r]\ar[u]^{\mm{0\\1}} & {0}
   }
\end{equation*}
Since the outer vertical maps are injective, so is the central
vertical map. It follows from a diagram chase that the image of a
generator of \( \K_0(C) \) in \( \GW^0(C,\OO(p)) \) cannot be
divisible by \( 2 \). Thus, the lower sequence must split, yielding \(
\GW^0(C,O(p))=\Z\oplus\Jac(C)[2] \).

Next, we see by \eqref{eq:FHcurvetwist} that the free summand in \( \GW^0(C,O(p)) \) has image \( \Z\cdot(2,1)\oplus 0 \) in \( \K_0(C) \). On \( \Jac(C)[2] \), on the other hand, \( F \) must be the inclusion \( \Jac(C)[2]\hookrightarrow\Jac(C) \), by comparison with the case of \( C-p \). This gives \( \GW^1(C,\OO(p))=\Z\oplus(\Jac(C)/\Jac(C)[2])\cong\Z\oplus\Jac(C) \).

The image of \( \GW^1(C,\OO(p)) \) in \( \K_0(C) \) may be read off directly from \eqref{eq:FHcurvetwist}, and we may deduce that \( \GW^2(C,\OO(p)) \) is a free abelian group generated by the twisted hyperbolic bundle \( H^2_{\OO(p)}(\OO) \). Equally directly, we see that \( \GW^3(C,\OO(p))=\Z\oplus \Jac(C) \). This completes the computation of all \GrothendieckWitt groups of \( C \).
\endproof

\encouragepagebreak{\subsectionair}
\subsection{KO-groups of curves}
Now suppose that \( C \) is a smooth curve defined over \( \C \). Its topological KO-groups may be computed from the Atiyah-Hirzebruch spectral sequence.
\begin{proposition}
Let \( C \) be a smooth projective complex curve of genus \( g \) as above. Then its KO-groups are given by
\begin{align*}
\KO^0(C) &= [\Z]\oplus(\Z/2)^{2g+1} &\KO^{-1}(C)&=[\Z/2]\oplus (\Z/2)^{2g} \\
\KO^2(C) &= \Z                      &\KO^1(C)&=\Z^{2g}\oplus\Z/2           \\
\KO^4(C) &= [\Z]                    &\KO^3(C)&=0                           \\
\KO^6(C) &= [\Z/2]\oplus\Z          &\KO^5(C)&=\Z^{2g}                       
\end{align*}
Here, the square brackets again indicate which summands vanish when passing to reduced groups.
\qed
\end{proposition}
In Section~\ref{sec:CS:Surfaces:KOK}, we will explain in detail how the values of the groups \( \tKOK{2i}{(C)} \) may also be obtained from the spectral sequence. They may be read off as special cases from Proposition~\ref{prop:KOK_Surfaces}. In fact, we can also read off the values of the twisted groups \( \tKOK{2i}{(C;\OO(p))} \) by applying the same proposition to the Thom space \( \Thom_C(\OO(p)) \) (see \cite{Me:WCCV}*{Section~2.1}). We then find that these groups agree with the corresponding Witt groups in Theorem~\ref{thm:W_Curve}:
\begin{align*}
\tKOK{0}{(C)} &= [\Z/2]\oplus(\Z/2)^{2g} &\tKOK{0}{(C,\OO(p))} &= (\Z/2)^{2g}    \\
\tKOK{2}{(C)} &= \Z/2                    &\tKOK{2}{(C,\OO(p))} &= 0  \\
\tKOK{4}{(C)} &= 0                       &\tKOK{4}{(C,\OO(p))} &= 0  \\
\tKOK{6}{(C)} &= 0                       &\tKOK{6}{(C,\OO(p))} &= 0  \\
\end{align*}

\encouragepagebreak{\sectionair}
\section{Surfaces}\label{sec:CS:Surfaces}
This section contains the computations of the Witt groups of smooth curves and surfaces. We also compute the corresponding topological groups \( \tKOK{2i}{} \). The comparison of the results is postponed until the next section.

\subsection{Witt groups of surfaces}\label{sec:CS:Surfaces:W}
Consider the algebraic K-group \( \K_0(X) \) of a smooth variety \( X \), and let \( c_i \) be the Chern classes \( c_i\colon\K_0(X)\rightarrow\CH^i(X) \) with values in the Chow groups of \( X \). Filter \( \K_0(X) \) by \( \K_0(X)\supset\rK_0(X)\supset\ker(c_1) \). If \( X \) has dimension at most two, then the map \( (\rank,c_1,c_2) \) induces an isomorphism \cite{Fulton:Intersection}*{Example 15.3.6}:
\begin{equation}\label{eq:grKalg}
 \gr^*(\K_0(X))\cong \Z\oplus\CH^1(X)\oplus\CH^2(X)
\end{equation}
A similar statement holds for the Witt group.
\begin{theorem}\label{thm:W_Surface}
Let \(X\) be a smooth variety over a field \(k\) of characteristic not two, such that the function field of \(X\) has \(2\)-cohomological dimension \(\cd_2(k(X))\leq 2\).  Filter \( \W^0(X) \) by \( \W^0(X)\supset \rW^0(X) \supset \ker(\bar{w}_1) \). Then the map \( (\rank,\bar{w}_1,\bar{w}_2) \) gives an isomorphism
\begin{align*}
 \gr^*(\W^0(X)) &\cong \Z/2 \oplus H^1_{\et}(X;\Z/2) \oplus \Big(\quotient{H^2_{\et}(X;\Z/2)}{\Pic(X)}\Big)
\intertext{Moreover, if we write \( S^1 \) for the squaring map
               \( \CH^1(X)/2 \rightarrow \CH^2(X)/2 \),
then the shifted Witt groups are as follows:}
  \W^1(X)      &\cong \ker(S^1)\oplus H^3_{\et}(X;\Z/2)\\
  \W^2(X)      &\cong \coker(S^1)\\
  \W^3(X)      &= 0
\end{align*}
\end{theorem}
\begin{remark}\label{rem:cd-assumptions}
As \(\cd_2(k(X))\leq \dim(X) + \cd_2(k)\), the condition of the theorem is satisfied in particular by surfaces over fields of \(2\)-cohomological dimension zero and by curves over fields of \(2\)-cohomological dimension at most one.  The condition \mbox{\(\cd_2(k)=0\)} is equivalent to \(k\) being hereditarily quadratically closed \cite{ElmanWadsworth}*{Lemma~2}.  Fields with \mbox{\(\cd_2(k)=1\)} include for example finite fields and function fields over algebraically closed fields.  Note that these fields have non-trivial Witt groups.
\end{remark}
\begin{example}\label{eg:projective-surface}
Suppose \( X \) is a smooth complex projective surface of Picard number \( \rho \). Write \( b_i \) for the Betti numbers of \( X \) and \( \nu \) for the rank of \( H^2(X;\Z)[2] \). Then the above result shows that the Witt groups of \( X \) are as follows:
\begin{align*}
  \W^0(X)  &= [\Z/2]\oplus (\Z/2)^{b_1+b_2-\rho + 2\nu}\\
  \W^1(X)  &= \begin{cases}
                (\Z/2)^{b_1+\rho + 2\nu} & \text{ if } S^1 = 0 \\
                (\Z/2)^{b_1+\rho + 2\nu - 1} & \text{ if } S^1 \neq 0
               \end{cases}\\
  \W^2(X)  &= \begin{cases}
                \Z/2  \quad\quad\quad\quad\quad\quad  & \text{ if } S^1 = 0 \\
                0                                     & \text{ if } S^1 \neq 0
               \end{cases}\\
  \W^3(X)  &= 0
\end{align*}
\end{example}
Our computation will follow a route described in \cite{Totaro:Witt}. Namely, there is a subtle relationship between Witt groups and {\'e}tale cohomology groups with \( \Z/2 \)-coefficients, encoded by three spectral sequences:
\begin{align*}
\label{ss:BO}\tag{BO} E_{2,\text{BO}}^{s,t}(X) &=
  H^s(X,\sheaf{H}^t)\quad\Rightarrow\quad H^{s+t}_{\et}(X;\Z/2)\\
\label{ss:P}\tag{P}   E_{2,\text{Par}}^{s,t}(X) &=
  H^s(X,\sheaf{H}^t)\quad\Rightarrow\quad H^s(X,\sheaf{W})\\
\label{ss:GW}\tag{GW} E_{2,\text{GW}}^{s,t}(X) &=
  H^s(X,\sheaf{W}^t)\quad\Rightarrow\quad \W^{s+t}(X)
\end{align*}
Here, \( \sheaf{H}^t \) is the Zariski sheaf attached to the presheaf that sends an open subset \( U\subset X \) to \( H^t_{\et}(U;\Z/2) \). Similarly, \( \sheaf{W} \) denotes the Zariski sheaf attached to the presheaf sending \( U \) to \( \W^0(U) \), and we set
\begin{equation*}
 \sheaf{W}^t=\begin{cases}
               \sheaf{W} & \text{for \( t\equiv 0 \mod 4 \)}\\
               0         & \text{otherwise}
             \end{cases}
\end{equation*}
Before specializing to the case of a surface, we give a short discussion of each of these spectral sequences for an arbitrary dimensional variety \( X \). For simplicity, we assume that \( X \) is a smooth variety over a field of characteristic not two throughout. More general statements may be found in the references.

\paragraph{The Bloch-Ogus spectral sequence (BO).} The first spectral sequence is the well-known Bloch-Ogus spectral sequence \cite{BlochOgus}. This is a first quadrant spectral sequence with differentials in the usual directions, \ie of bidegree \( (r,-r+1) \) on the \( r^{\text{th}} \) page. The \( E_2 \)-page is concentrated above the main diagonal \( s=t \), and the groups along the diagonal may be identified with the Chow groups of \( X \) modulo two \cite{BlochOgus}*{Corollary~6.2 and proof of Theorem~7.7}:
\shortendisplayskip
\begin{equation}\label{eq:Hs-of-Hs_Chow}
  H^s(X,\sheaf{H}^s)\cong \CH^s(X)/2
\end{equation}
The exactness of the Gersten complex for ^^e9tale cohomology [12, \S~2.1, Thm.~4.2] implies that the sheaves \(\sheaf H^t\) vanish for all \(t>\cd_2(k(X))\).  Thus, the spectral sequence is concentrated in rows \( 0\leq t\leq \cd_2(k(X)) \).

\begin{figure}[!b]
\begin{equation*}
 \newcommand{\Legende}{\xymatrix{a\ar[r]&b}}
 \newcommand{\Plathalter}{\phantom{wwww}\cdots\phantom{wwww}}
 \xymatrix@C=1pt@R=0pt{
 && \\
 && {\Plathalter}               && {\Plathalter}        && {\Plathalter}   && {\Plathalter}  & \\
 && && && && & \\
 && {H^0(X,\sheaf{H}^3)}\ar[rruu]\ar@{--}[ruu] && {H^1(X,\sheaf{H}^3)}\ar[rruu]\ar@{--}[ruu]    && {H^2(X,\sheaf{H}^3)}\ar[rruu]\ar@{--}[ruu] && {\CH^3(X)/2}\ar@{-}[ru]\ar@{--}[ruu]  & \\
 \\
 && {H^0(X,\sheaf{H}^2)}\ar[rruu]\ar@{-->}[rruuuu] && {H^1(X,\sheaf{H}^2)}\ar[rruu]\ar@{-->}[rruuuu]    && {\CH^2(X)/2}\ar[rruu]\ar@{-->}[rruuuu] && {\cdot}  \\
 \\
 && {H^0(X,\sheaf{H}^1)}\ar[rruu]\ar@{-->}[rruuuu]\ar@{..>}[rruuuuuu] && {\CH^1(X)/2}\ar[rruu]\ar@{-->}[rruuuu]\ar@{..>}[rruuuuuu]    && {\cdot}   && {\cdot}  \\
 \\
 && {\CH^0(X)/2}\ar[rruu]|(.55){d_2}\ar@{-->}[rruuuu]|(.46){d_3}\ar@{..>}[rruuuuuu]|(.35){d_4} && {\cdot}     && {\cdot}   && {\cdot}  \\
 \ar[rrrrrrrrr]_>{s} &&&&&&&&&\\
 &\ar[uuuuuuuuuuu]^>{t}
 }
\end{equation*}
\caption{Pardon's spectral sequence. The entries below the diagonal are zero.}
\label{fig:PSS}
\end{figure}
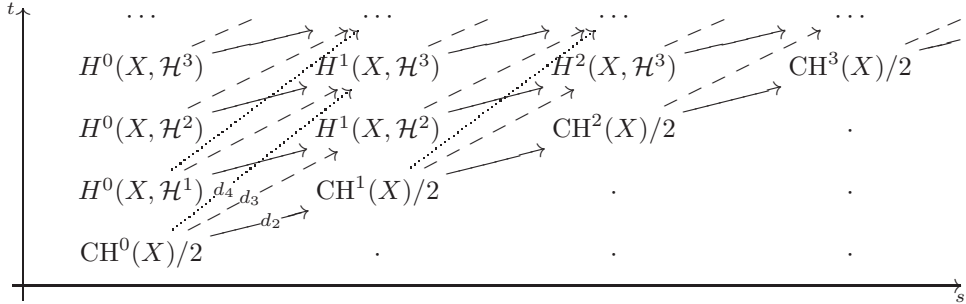

\paragraph{Pardon's spectral sequence (P).} Pardon's spectral sequence can be indexed to have the same \( E_2 \)-page as the Bloch-Ogus spectral sequence. The differentials on the \( r^{th} \) page then have bidegree \( (1,r-1) \), as illustrated in Figure~\ref{fig:PSS}. In \cite{Totaro:Witt}, Totaro shows that, under the identifications \eqref{eq:Hs-of-Hs_Chow},
the differential \( d_2 \) on the main diagonal of the \( E_2 \)-page corresponds to the Steenrod operation \( S^1\colon{\CH^s(X)/2\rightarrow\CH^{s+1}(X)/2} \),
as defined by Brosnan and Voevodsky \citelist{\cite{Brosnan:Steenrod}\cite{Voevodsky:RPO}}.

The sequence converges to the cohomology of \( X \) with respect to \( \sheaf{W} \)  in the usual sense that the \( i^{\text{th}} \) column of the \( E_{\infty} \)-page is isomorphic to the associated graded module of \( H^i(X,\sheaf{W}) \) with respect to some filtration. In order to describe the filtration on \( H^0(X,\sheaf{W})=\sheaf{W}(X) \), we briefly summarize how the spectral sequence arises.

\newcommand{\complex}[1]{\mathbb{#1}}
Let \( X \) be as above, and let \( K \) be its function field. We denote the residue field of a scheme-theoretic point \( x \) of \( X \) by \( k(x) \). In \cite{Pardon}, Pardon shows that the sheaf \( \sheaf{W} \) has a flasque resolution by a Gersten-Witt complex \( \complex{W} \) of the form
\begin{equation}\label{eq:GWcomplex}
 \complex{W}\colon{}\quad\W^0(K)\rightarrow\coprod_{x\in X^{(1)}}\W^0(k(x))\rightarrow\cdots\rightarrow\coprod_{x\in X^{(n)}}\W^0(k(x))\rightarrow 0
\end{equation}
Here, \( X^{(i)} \) denotes the set of codimension \( i \) points of \( X \), and \( \W^0(k(x)) \) is to be viewed as a constant sheaf supported on the closure of \(  x  \) in \( X \).
In particular, \( \sheaf{W} \) is a subsheaf of the constant sheaf \( \W^0(K) \). Let \( I(K)\subset \W^0(K) \) be the fundamental ideal of \( K \) (see Section~\ref{sec:CS:SW}). The fundamental filtration on \( \W^0(K) \), given by the powers of \( I(K) \), induces a filtration on \( \sheaf{W} \) which we denote by
\begin{equation}\label{eq:def:sheafI}
   \sheaf{I}_t:=\sheaf{W}\cap I^t(K)
\end{equation}
Pardon shows more generally that the fundamental filtrations of the Witt groups \( \W^0(k(x)) \) give rise to flasque resolutions of the sheaves \( \sheaf{I}_t \) of the form
\begin{align*}
 \mathbb{I}_t\colon{}\quad I^t(K)\rightarrow \coprod_{x\in X^{(1)}}I^{t-1}(k(x)) \rightarrow
 \coprod_{x\in X^{(2)}}I^{t-2}(k(x))\rightarrow \dots
\end{align*}
Applying the standard construction of the spectral sequence of a filtered complex of abelian groups \cite{GelfandManin}*{Chapter~III.7, Section~5} to the filtration of \( \complex{W}(X) \) by \( \complex{I}_t(X) \), we obtain a spectral sequence
\begin{equation*}
H^s(X,\sheaf{I}_t/\sheaf{I}_{t+1}) \Rightarrow H^s(X,\sheaf{W})
\end{equation*}
To conclude, Pardon uses the affirmation of the Milnor conjectures (see \eqref{eq:Milnors-isos}) to identify the sheaves \( \sheaf{I}_t/\sheaf{I}_{t+1} \) with the sheaves \( \sheaf{H}^t \). The following lemma is a direct consequence of this construction.
\begin{lemma}\label{lem:Pardons_E_infty}
 The filtration on \( \sheaf{W}(X) \) appearing on the zeroth column of the \( E_{\infty} \)-page of Pardon's spectral sequence is given by the global sections of the sheaves \( \sheaf{I}_t \) defined in \eqref{eq:def:sheafI}:
 \begin{equation*}\clubpenalty10000%
 E_{\infty,\text{Par}}^{0,t}(X)\cong \sheaf{I}_t(X)/\sheaf{I}_{t+1}(X)
 \end{equation*}
The edge homomorphism including this column into the zeroth column of the \( E_2 \)-page is induced by the sheafification map from the presheaf quotient of \( \sheaf{I}_t \) by \( \sheaf{I}_{t+1} \) to the quotient sheaf \( \sheaf{I}_t/\sheaf{I}_{t+1} \).
\end{lemma}
The group \( \sheaf{W}(X) \), viewed as a subgroup of \( \W^0(K) \), is commonly referred to as the unramified Witt group of \(  X  \).

\paragraph{The Gersten-Witt spectral sequence (GW).}
The Gersten-Witt spectral sequence of Balmer and Walter, third on the list above, may be thought of as an analogue of the Bloch-Ogus spectral sequence for Witt groups. In \cite{BalmerWalter:GWSS}, Balmer and Walter, too, construct  a Gersten-Witt complex of the form \eqref{eq:GWcomplex}. Although it seems likely that their construction agrees with Pardon's, this does not yet seem to have been established (\cf \cite{Gille:gradedGW}*{Introduction}).  It is known in any case, however, that their complex also provides a flasque resolution of \( \sheaf{W} \) \cite{BGPW:Gersten}*{Lemma~4.2 and Theorem~6.1}. 
 The spectral sequence constructed in \cite{BalmerWalter:GWSS} may therefore be rewritten in the form above:
\begin{theorem}[\cite{BalmerWalter:GWSS}]
 Let \( X \) be a smooth variety over a field of characteristic not two. There is a spectral sequence with \( E_2 \)-page
 \begin{equation*}
  E_2^{s,t}=\begin{cases}
             H^s(\sheaf{W}) & \text{ in rows } t\equiv 0\mod 4 \\
             0              & \text{ elsewhere }
            \end{cases}
 \end{equation*}
 The sequence has differentials \( d_r \) of bidegree \( (r,-r+1) \) and converges to \( \W^{s+t}(X) \).
\end{theorem}
\proof[Proof of Theorem~\ref{thm:W_Surface}]
We now specialize to the case when \(\cd_2(k(X))\leq 2\).  Then the Bloch-Ogus spectral sequence \eqref{ss:BO} collapses immediately and we obtain descriptions of the groups \( H^s(X;\sheaf{H}^t) \) in terms of the {\'e}tale cohomology of \( X \) with \( \Z/2 \)-coefficients.

The \( E_2 \)-page of Pardon's spectral sequence therefore takes the following form:
\begin{equation*}
 \newcommand{\Plathalter}{\phantom{wwwww}\phantom{\cdot}\phantom{wwwww}}
 \xymatrix@C=1pt@R=0pt{
 && {\Plathalter}                           && {\Plathalter}        && {\Plathalter}       &\\
 && {\quotient{H^2_{\et}(X;\Z/2)}{\Pic(X)}} && {H^3_{\et}(X;\Z/2)}  && {H^4_{\et}(X;\Z/2)} &\\
 &&                                         &&                      && {}\save[]+<0.3cm,0.1cm>*{\cong \CH^2(X)/2} \restore & \\
 && {H^1_{\et}(X;\Z/2)}\ar[rruu]|{d_2}      && {\Pic(X)/2}\ar[rruu]|{S^1} && {\cdot}   &  \\
 \\
 && {\Z/2}                                  && {\cdot}              && {\cdot}   &  \\
 \ar[rrrrrrr]_>{s} &&&&&&&\\
 &\ar[uuuuuuu]^>{t}
 }
\end{equation*}
On the other hand, we see that the Gersten-Witt spectral sequence also collapses, so that \( \W^i(X)\cong H^i(X,\sheaf{W}) \). Thus, the \( i^{\text{th}} \) column of Pardon's spectral sequence simply converges to \( \W^i(X) \). In particular, if we write \( I_t(X) \) for the filtration on \( \W^0(X) \) corresponding to the filtration \( \sheaf{I}_t(X) \) of \( \sheaf{W}(X) \) defined in \eqref{eq:def:sheafI}, then by Lemma~\ref{lem:Pardons_E_infty} we have \( E_{\infty, \text{Par}}^{0,t}\cong \gr_t^{I}\W^0(X) \).

\begin{lemma}
Let \( X \) be as above. Then, for \( t=1 \) or \( 2 \), the edge homomorphisms
\[ \gr_t^{I}\W^0(X)\cong E_{\infty, \text{Par}}^{0,t}\hookrightarrow E_{2, \text{Par}}^{0,t} \] can be identified with the Stiefel-Whitney classes \( \bar{w}_t \). More precisely, \( I_1(X)=\rW^0(X) \), \( I_2(X)=\ker(\bar{w}_1) \), and we have commutative diagrams
\begin{align*}
\xymatrix@R=6pt{
   {E_{\infty, \text{Par}}^{0,1}} \ar@{^{(}->}[r]\ar@{=}[d] & {E_{2, \text{Par}}^{0,1}} \ar@{=}[d]\\
   {\quotient{\rW^0(X)}{\ker(\bar{w}_1)}} \ar[r]^-{\bar{w}_1} & {H^1_{\et}(X;\Z/2)}
\\
\\
   {E_{\infty, \text{Par}}^{0,2}} \ar@{^{(}->}[r]\ar@{=}[d] & {E_{2, \text{Par}}^{0,2}}\ar@{=}[d] \\
   {\ker(\bar{w}_1)} \ar[r]^-{\bar{w}_2} & {\quotient{H^2_{\et}(X;\Z/2)}{\Pic(X)}}
}
\end{align*}
\end{lemma}
\begin{proof}
The statement of the lemma is not surprising: the only non-obvious maps that enter into the construction of Pardon's spectral sequence are the isomorphisms \( e^t \) displayed in \eqref{eq:Milnors-isos};
for \( t=1 \) or \(2 \), these can be identified with Stiefel-Whitney classes, and the latter can be defined globally.

In more detail, the various identifications arising from the Bloch-Ogus spectral sequence and Pardon's spectral sequence fit into the following diagram:
\begin{equation*}
\xymatrix{
  {E^{0,t}_{\infty, \text{Par}}} \ar@{^{(}->}[r]\ar@{=}[d] &
  {E^{0,t}_{2,\text{Par}}}\ar@{=}[d]
  \\
  {\quotient{I_t(X)}{I_{t+1}(X)}}\ar[r]
      \ar@{..>}[d]|{\text{\rule{0pt}{10pt}\raisebox{2pt}{?}}} &
  {\left(\quotient{\sheaf{I}_t}{\sheaf{I}_{t+1}}\right)(X)} \ar@{^{(}->}[r] \ar[d]^{\cong} &
  {\quotient{I^t(K)}{I^{t+1}(K)}} \ar[d]^{\cong}_{e^t}
  \\
  {H^t_{\et}(X;\Z/2)} \ar[r]^-{\text{sheafification}}\ar@{->>}[d] &
  {\sheaf{H}^t(X)} \ar@{^{(}->}[r] &
  {H^t_{\et}(K;\Z/2)}
  \\
  {E^{0,t}_{\infty, \text{BO}}} \ar[r]^{\cong}_{\text{(BO collapses)}} &
  {E^{0,t}_{2,\text{BO}}}\ar@{=}[u]
}
\end{equation*}
The claim is that, for \( t=1 \) or \( 2 \), the diagram commutes if we take the dotted map to be the \( t^{\text{th}} \) Stiefel-Whitney class. Indeed, \( e^t \) can be identified with \( \bar{w}_t \) in these cases, and the horizontal compositions across the two middle lines in the diagram are given by pullback along the inclusion of the generic point into \( X \).
\end{proof}

There are two potentially non-zero differentials in Pardon's spectral sequence: the differential \( S^1\colon\CH^1(X)/2\rightarrow \CH^2(X)/2 \), which is simply the squaring operation, and another differential \( d \) from \( H^1_{\et}(X;\Z/2) \) to \( H^3_{\et}(X;\Z/2) \).

Since the first Stiefel-Whitney class always surjects onto \(
H^1_{\et}(X;\Z/2) \), we see from the previous lemma that \(
H^1_{\et}(X;\Z/2) \) survives to the \( E_{\infty} \)-page. So the
differential \( d \) must vanish. On the other hand, since \(
E_2^{0,2}=E_{\infty}^{0,2} \), the lemma implies that the restriction
of the second Stiefel-Whitney class to the kernel of the first gives
an isomorphism with \( H^2_{\et}(X;\Z/2)/\Pic(X) \).
This proves the first statement of the proposition. Since \( S^1 \) is the only non-trivial differential left, the values of the shifted Witt groups may also be read off the spectral sequence. This completes the proof of Theorem~\ref{thm:W_Surface}.
\endproof

\begin{corollary}\label{cor:etale-w2_Surface}
  Let \(X\) be as in Theorem~\ref{thm:W_Surface}.  If we filter \(\GW^0(X)\) by \(\GW^0(X)\supset \rGW^0(X) \supset \ker(w_1) \supset \ker(w_2)\), we obtain an isomorphism
  \begin{equation*}
    \gr^* \GW^0(X) \cong \Z \oplus H^1_\et(X;\Z/2) \oplus H^2_\et(X;\Z/2) \oplus \CH^2(X).
  \end{equation*}
\end{corollary}

\begin{proof}
  It suffices to show that the second Stiefel-Whitney class restricts to an epimorphism
  \(
  \ker(w_1)\longtwoheadrightarrow H^2_{\et}(X;\Z/2)
  \)
  with kernel the isomorphic image of \(\ker(c_1|_{\rK_0(X)})\cong\CH^2(X)\) under the hyperbolic map.

  To see this, note first that the hyperbolic map \( H^0\colon{\rK_0(X)\rightarrow\rGW^0(X)} \) factors through the kernel of \( w_1 \).
  In fact, we can restrict the Karoubi sequence to an exact sequence forming the first row of the following diagram.
  \begin{equation*}
    \xymatrix{
      {0} \ar[r]  &
      {F(\rGW^3(X))} \ar[r] \ar[d] &
      {\rK_0(X)}             \ar[r]^{H^0}   \ar@{->>}[d]^{c_1} &
      {\ker(w_1)}       \ar[r]   \ar[d]^{w_2} &
      {\ker(\bar{w}_1)} \ar[r]\ar[d]^{\bar{w}_2}_{\cong} &
      {0} \\
      {0} \ar[r] &
      {\Pic(X)^{\otimes 2}} \ar[r] &
      {\Pic(X)}               \ar[r]           &
      {H^2_{\et}(X;\Z/2)}     \ar[r]     &
      {\quotient{H^2_{\et}(X;\Z/2)}{\Pic(X)}}  \ar[r] &
      {0}
    }
  \end{equation*}
  The lower row is obtained from the Kummer sequence, hence also exact, and the diagram commutes.  We know from the last paragraph of the proof of Theorem~\ref{thm:W_Surface} that the map \( \bar{w}_2 \) on the right is an isomorphism. Since \( c_1 \) is surjective, the claimed surjectivity of \( w_2 \) follows.

  It remains to show that the restriction of \(H^0\) to the kernels \(\ker(c_1)\rightarrow \ker(w_2) \)  is an isomorphism.  By an instance of the Snake Lemma, this is equivalent to showing that \(c_1\) maps \(F(\rGW^3(X))\) isomorphically to \( \Pic(X)^2 \).  As \(\W^3(X)\) is zero, \(F(\rGW^3(X))\) coincides with \( FH^3(\rK_0(X)) \).  The claim therefore follows from the next lemma.
\end{proof}

\begin{lemma}
  Consider the endomorphism \(FH^i\colon \K_0(X)\to \K_0(X)\) (for an arbitrary scheme \(X\) over \(\Z[\frac{1}{2}]\)). For any \(x\in\K_0(X)\), we have
  \begin{align*}
    &\left.
      \begin{aligned}
        c_1(FH^ix) &= 0\\
        c_2(FH^ix) &= 2c_2(x)-c_1(x)^2
      \end{aligned}\quad\quad\quad
    \right\}\text{ if \(i\) is even,}\\
    &\left.
      \begin{aligned}
        c_1(FH^ix) &=2c_1(x) \\
        c_2(FH^ix) &= 2c_1(x)^2
        = c_1(FH^ix)\cdot c_1(x)
      \end{aligned}\quad
    \right\}\text{ if \(i\) is odd.}
  \end{align*}
\end{lemma}
\begin{proof}
  This follows from the equations \(FH^i(x)=x+(-1)^ix^\vee\) and \(c_i(x^\vee) = (-1)^i c_i(x)\) \cite{Fulton:Intersection}*{Remark~3.2.3(a)}.
\end{proof}

\encouragepagebreak{\subsectionair}
\subsection{KO/K-groups of surfaces}\label{sec:CS:Surfaces:KOK}
On a topological space \( X \), let us write \( \Sq^2_{\Z} \) for the composition
\begin{equation*}
   H^2(X;\Z)\!/2\;\hookrightarrow H^2(X;\Z/2) \overset{\Sq^2}\longrightarrow H^4(X;\Z/2)
\end{equation*}
where \( \Sq^2 \) is the squaring operation.
\begin{proposition}\label{prop:KOK_Surfaces}
Let \( X \) be a connected CW complex of dimension at most four. Filter the group \( (\KO^0\!/\K)(X) \) by
\( (\KO^0\!/\K)(X)\supset(\rKO^0\!/\rK)(X)\supset\ker(\bar{w}_1) \).
Then the map \( (\rank, \bar{w}_1,\bar{w}_2) \) induces an isomorphism
\begin{align*}
 \gr\left(\tKOK{0}{(X)}\right) &\cong \Z/2\oplus H^1(X;\Z/2) \oplus \left(\quotient{H^2(X;\Z/2)}{H^2(X;\Z)}\right)
\intertext{The remaining groups \( \tKOK{2i}{(X)} \) have the following values:}
 \tKOK{2}{(X)} &\cong \ker(\Sq^2_{\Z}) \oplus H^3(X;\Z/2) \\
 \pagebreak[0]
 \tKOK{4}{(X)} &\cong \coker(\Sq^2_{\Z})\\
 \tKOK{6}{(X)} &= 0
\end{align*}
\end{proposition}
\begin{example}
 Suppose \( X \) is a compact four-dimensional manifold. Write \( b_i \) for its Betti numbers and \( \nu \) for the rank of \( H^2(X;\Z)[2] \). Then the above result shows that
\pagebreak[0]
\begin{align*}
 \tKOK{0}{(X)} &= [\Z/2]\oplus (\Z/2)^{b_1+2\nu}\\
 \tKOK{2}{(X)} &= \begin{cases}
                (\Z/2)^{b_1+b_2+2\nu}   & \text{ if } \Sq^2_{\Z} = 0 \\
                (\Z/2)^{b_1+b_2+2\nu-1} & \text{ if } \Sq^2_{\Z}\neq 0
                \end{cases}\\
 \tKOK{4}{(X)} &= \begin{cases}
                \Z/2  \quad\quad\quad\quad\quad\quad & \text{ if } \Sq^2_{\Z} = 0 \\
                0                                    & \text{ if } \Sq^2_{\Z}\neq 0
                \end{cases}\\
 \tKOK{6}{(X)} &= 0
\end{align*}
\end{example}

\begin{remark}
  Note that both cases ---~\( \Sq^2_{\Z}=0 \) and \( \Sq^2_{\Z} \) onto~--- can occur even for complex projective surfaces of geometric genus zero.
  For example, \( \Sq^2_{\Z} \) is onto for \( \P^2 \). On the other hand, the Wu formula \cite{MilnorStasheff}*{Theorem~11.14} shows that \( \Sq^2 \) is given by multiplication with \( c_1(X) \) mod \( 2 \). So \( \Sq^2_{\Z}=0 \) on any projective surface \( X \) of geometric genus zero
whose canonical divisor is numerically trivial. Concretely, we could take \( X \) to be an Enriques surface (a quotient of a K3-surface by a fixed-point free involution): see \cite{Beauville}*{page~90}.
\end{remark}

To calculate the groups \( \tKOK{2i}{(X)} \), we use the Atiyah-Hirzebruch spectral sequences for K- and KO-theory:
\begin{align*}
 &E^{p,q}_2=H^p(X;\K^q(\point))\Rightarrow \K^{p+q}(X)\\
 &E^{p,q}_2=H^p(X;\KO^q(\point))\Rightarrow \KO^{p+q}(X)
\end{align*}
For K-theory, the spectral sequence has the singular cohomology of \(
X \) with integral coefficients in all even rows, while the odd rows
vanish. The spectral sequence for KO-theory is \( 8 \)-periodic in \( q \). We have the integral cohomology of \( X \) in rows \( q\equiv 0 \) and \( q\equiv-4 \) mod~\( 8 \), its cohomology with \( \Z/2 \)-coefficients in rows \( q\equiv -1 \) and \( q\equiv -2 \), and all other rows are zero.
\begin{figure}[h!]
\begin{equation*}
\xymatrix@R=0pt@C=0pt{
 & \ar@{<-}[ddddddddddddd]_<{q}\\
 && {H^0(X;\Z)} && {H^1(X;\Z)} && {H^2(X;\Z)}\ar[rrrrdd]|{\;{\Sq^2}\circ\pi} && {H^3(X;\Z)} && {H^4(X;\Z)} \\
 \ar[rrrrrrrrrrrr]_>{\;p}|(.69){\;\phantom{{\Sq^2}\circ\pi}\;\;} && && && && && && \\
 && {H^0(X;\Z/2)} && {H^1(X;\Z/2)} && {H^2(X;\Z/2)}\ar[rrrrdd]|{\;\Sq^2} && {H^3(X;\Z/2)} && {H^4(X;\Z/2)} \\
 &  \\
 && {H^0(X;\Z/2)} && {H^1(X;\Z/2)} && {H^2(X;\Z/2)} && {H^3(X;\Z/2)} && {H^4(X;\Z/2)} \\
 &  \\
 && {0} && {0}  && {0} && {0} && {0}\\
 & \\
 && {H^0(X;\Z)} && {H^1(X;\Z)} && {H^2(X;\Z)} && {H^3(X;\Z)} && {H^4(X;\Z)} \\
 & \\
 && {0} && {0} && {0} && {0} && {0} \\
 & \\
 && {\vdots} && {\vdots} && {\vdots} && {\vdots} && {\vdots}\\
 }
\end{equation*}
\caption{The \(E_2\)-page of the Atiyah-Hirzebruch spectral sequence computing \(\KO^*(X)\)}
\label{fig:KO-AHSS}
\end{figure}

The low-order differentials of these spectral sequences may be described as follows. Let
\begin{equation*}
\begin{aligned}
 \pi&\colon H^*(X;\Z)\rightarrow     H^*(X;\Z/2)\\
 \Sq^2&\colon H^*(X;\Z/2)\rightarrow   H^{*+2}(X;\Z/2)\\
 \beta&\colon H^*(X;\Z/2)\rightarrow   H^{*+1}(X;\Z)
\end{aligned}
\end{equation*}
denote reduction modulo two, the second Steenrod square, and the
Bockstein homomorphism, respectively. For K-theory, the differentials
on the \(E_3\)-page of the Atiyah-Hirzebruch spectral sequence are
given by the composition \(\beta\circ{\Sq^2}\circ\pi\)
\cite{HJJS:BasicBundleTheory}*{Chapter~21, \S~5}. For KO-theory, the
differentials \( d_2^{*,0} \) and \( d_2^{*,-1} \) on the \(E_2\)-page
are given by \( \Sq^2\circ\pi \) and \( \Sq^2 \), respectively. The
differential \( d_3^{*,-2}\) on the \(E_3\)-page can be identified
with \( \beta\circ \Sq^2 \) \cite{Fujii:P}*{1.3}. The following lemma is now immediate.
\begin{lemma}\label{lem:AHSS_d3}
The differential
\(
 d_3^{1,-2}\colon{E_3^{1,-2}\longrightarrow E_3^{4,-4}}
\)
is trivial both in the Atiyah-Hirze\-bruch spectral sequence for K-theory and in the spectral sequence for KO-theory.
\end{lemma}

\begin{lemma}\label{lem:c-and-w-in-AHSS}%
Let \( X \) be a connected finite-dimensional CW complex. Denote the filtrations on the groups \( \K^0(X) \) and \( \KO^i(X) \) associated with the Atiyah-Hirzebruch spectral sequences by
\begin{equation*}
\begin{aligned}
   & \K_0(X) \supset \K_1(X) = \K_2(X) \supset \K_3(X) = \K_4(X) \supset \dots \\
   & \KO^i_0(X) \supset \KO^i_1(X) \supset \KO^i_2(X) \supset \KO^i_3(X) \supset \dots
\end{aligned}
\end{equation*}
The initial layers of these filtrations have more intrinsic descriptions in terms of Chern and Stiefel-Whitney classes:
\begin{itemize}
 \item \( \K_2(X)=\K_1(X)=\rK^0(X) \), and \( \KO^i_1(X)=\rKO^i(X) \)
 \item The map \( \K_2(X)\rightarrow H^2(X;\Z) \) arising from the spectral sequence may be identified with the first Chern class, at least up to a sign. That is, we have the following commutative diagram:
 \begin{equation*}
  \xymatrix@R=10pt{
   {E^{2,-2}_{\infty}} \ar@{^{(}->}[r] & {E_2^{2,-2}}\\
   {\rK^0(X)} \ar@{->>}[u] \ar@{->>}[r]^-{\pm c_1} & {H^2(X;\Z)}\ar@{=}[u]
  }
 \end{equation*}
 In particular, \( \K_4(X)=\ker(c_1) \). Moreover, since \( c_1 \) is surjective, \( E^{2,-2}_{\infty} = E^{2,-2}_2\).
 \item The map \( \K_4(X)\rightarrow H^4(X;\Z) \) can be identified with the restriction of the second Chern class to \(\ker(c_1)\), again up to a sign:
 \begin{equation*}
  \xymatrix@R=10pt{
   {E^{4,-4}_{\infty}} \ar@{^{(}->}[r] & {E_2^{4,-4}}\\
   {\ker(c_1)} \ar@{->>}[u] \ar[r]^-{\pm c_2} & {H^4(X;\Z)}\ar@{=}[u]
  }
 \end{equation*}
 Thus, \( \K_6(X) = \ker(c_2) \). (Note that the upper horizontal map in this diagram is indeed an inclusion, by the previous lemma.)
 \item Likewise, the map \( \KO^0_1(X)\rightarrow H^1(X;\Z/2) \) can be identified with the first Stiefel-Whitney class:
 \begin{equation*}
  \xymatrix@R=10pt{
   {E^{1,-1}_{\infty}} \ar@{^{(}->}[r] & {E_2^{1,-1}}\\
   {\rKO^0(X)} \ar@{->>}[u] \ar@{->>}[r]^-{w_1} & {H^1(X;\Z/2)}\ar@{=}[u]
  }
 \end{equation*}
 Thus, \( \KO^0_2(X)=\ker(w_1) \) and \( E^{1,-1}_\infty = E_2^{1,-1} \).

 \item Lastly, the map \( \KO^0_2(X)\rightarrow H^2(X;\Z/2) \) can be identified with the restriction of the second Stiefel-Whitney class, \ie we have a commutative diagram
 \begin{equation*}
 \xymatrix@R=10pt{
   {E^{2,-2}_{\infty}} \ar@{^{(}->}[r] & {E_2^{2,-2}}\\
   {\ker(w_1)} \ar@{->>}[u] \ar[r]^-{w_2} & {H^2(X;\Z/2)}\ar@{=}[u]
  }
 \end{equation*}
 Thus, \( \KO^0_4(X)=\ker(w_2) \).
\end{itemize}
\end{lemma}

\begin{proof}
The first statement is clear. The other assertions follow by viewing
the maps in question as cohomology operations and computing them for a
few spaces. For lack of a reference, we include a more detailed proof of the statements concerning KO-theory. The case of complex K-theory can be dealt with analogously.

First, we analyse the map \( \rKO^0(X)\rightarrow H^1(X;\Z/2) \) arising from the spectral sequence. A priori, we have defined this map only for finite-dimensional CW complexes. But we can extend it to a natural transformation of functors on the homotopy category of all connected CW complexes, using the fact that the canonical map \(H^1(X;\Z/2) \rightarrow \lim_i H^1(X^i;\Z/2) \) is an isomorphism for any CW complex \( X \) with \(i\)-skeletons \( X^i \). On the homotopy category of connected CW complexes, the functor \( \rKO^0(-) \) is represented by \( \mathrm{BO} \). Natural transformations \( \rKO^0(-)\rightarrow H^1(-;\Z/2) \) are therefore in one-to-one correspondence with elements of \( H^1(\mathrm{BO};\Z/2)=\Z/2\cdot w_1 \), where \( w_1 \) is the first Stiefel-Whitney class of the universal bundle over \( \mathrm{BO} \). Thus, either the map in question is zero, or it is given by \( w_1 \) as claimed. Since it is non-zero on \( S^1 \), the first case may be discarded.

To analyse the map \( \KO^0_2(X)\rightarrow H^2(X;\Z/2) \), we note that the previous conclusion yields a functorial description of \( \KO^0_2(X) \) as the kernel of \( w_1 \) on \( \rKO^0(X) \). Moreover, given this description, we may define a natural set-wise splitting of the inclusion of \( \KO^0_2(X) \) into \( \rKO^0(X) \) for any finite-dimensional CW complex \(X\) as follows:
\begin{align*}
 \rKO^0(X)\twoheadrightarrow &\;\KO^0_2(X)\\
 [\vb{E}]-\rank(\vb{E})[\R] \mapsto &\;[\vb{E}]-\rank(\vb{E})[\R] - [\det\vb E] + [\R]
\end{align*}
The composition \( \rKO^0(-)\rightarrow H^2(-;\Z/2) \) can be extended to a natural transformation of functors on the homotopy category of connected CW complexes in the same way as before, and it may thus be identified with an element of \( H^2(\mathrm{BO};\Z/2)=\Z/2\cdot w_1^2\oplus\Z/2\cdot w_2 \). Consequently, its restriction
to \( \KO^0_2(X) \) is either \( 0 \) and \( w_2 \). Again, the first possibility may be discarded, for example by considering \( S^2 \).
\end{proof}

Lemma~\ref{lem:c-and-w-in-AHSS} has some immediate implications for low-dimensional spaces. In the following corollary, all characteristic classes refer to the corresponding maps on reduced groups:
\begin{alignat*}{2}
   c_i&\colon &\rK^0(X)  \longrightarrow &\;H^{2i}(X;\Z)\\
   w_i&\colon &\rKO^0(X) \longrightarrow &\;H^i(X;\Z/2)
\end{alignat*}
\begin{corollary}\label{cor:w2_Surface}
Let \( X \) be a connected CW complex of dimension at most four. Then the second Chern class is surjective and restricts to an isomorphism
\begin{alignat*}{2}
  c_2&\colon &\ker(c_1)\overset{\cong}\longrightarrow &\;H^4(X;\Z)
\intertext{%
The second Stiefel-Whitney class is also surjective, and restricts to an epimorphism
}
  w_2&\colon &\ker(w_1)\longtwoheadrightarrow &\;H^2(X;\Z/2)
\intertext{%
Its kernel is given by the image of \( \ker(c_1) \) under realification. Moreover, the induced map \( \bar{w}_2 \) on \( \trKOK{0}{(X)} \) restricts to an isomorphism
}
  \bar{w}_2 &\colon &\ker(\bar{w}_1)\overset{\cong}\longrightarrow &\;\left(\quotient{H^2(X;\Z/2)}{H^2(X;\Z)}\right)
\end{alignat*}

\end{corollary}
\begin{proof}
Consider the Atiyah-Hirzebruch spectral sequences for K- and KO-theory. \\
\noindent Lemma~\ref{lem:AHSS_d3} shows that the sequence for K-theory collapses. Moreover, in the spectral sequence for KO-theory, no differentials affect the diagonal computing \( \KO^0(X) \). This implies the first two claims of the corollary. Next, we compare the two spectral sequences via realification. The description of the kernel of \( w_2 \) may be obtained from the following row-exact commutative diagram:\pagebreak[0]
\begin{equation*}
\xymatrix{
    {0}      \ar[r]
  & {H^4(X;\Z)} \ar[r]\ar[d]^{\cong}
  & {\rK(X)} \ar[r]^{c_1}\ar[d]^{r}
  & {H^2(X;\Z)} \ar[r]\ar[d]^{\text{reduction mod 2}}
  & {0}
  \\ {0}          \ar[r]
  & {H^4(X;\Z)}      \ar[r]
  & {\ker(w_1)}   \ar[r]^{w_2}
  & {H^2(X;\Z/2)} \ar[r]
  & {0}
  }
\end{equation*}
The final claim of the corollary also follows from this diagram, by identifying
 \( \ker(\bar w_1) \) with \(\ker(w_1)/\rK(X) \).
\end{proof}
The situation for a connected CW complex of dimension at most four may now be summarized as follows. Firstly, by Lemma~\ref{lem:c-and-w-in-AHSS}, the filtrations on the groups \(\K^0(X)\) and \(\KO^0(X)\) arising in the Atiyah-Hirzebruch spectral sequences can be written as
\begin{align*}
 \K^0(X)  &\supset \rK^0(X) \supset \ker(c_1) \\
 \KO^0(X) &\supset \rKO^0(X) \supset \ker(w_1)
\end{align*}
Secondly, Corollary~\ref{cor:w2_Surface} implies that the maps \( (\rank, c_1, c_2) \) and \( (\rank, \bar{w}_1, \bar{w}_2) \) on the associated graded groups induce isomorphisms
\begin{align}
  \gr(\K^0(X))&\cong \;\Z\;\oplus H^2(X;\Z) \;\oplus\; H^4(X;\Z) \label{eq:grKtop}\\
  \gr\left(\tKOK{0}{(X)}\right)&\cong \Z/2\oplus H^1(X;\Z/2) \oplus \left(\quotient{H^2(X;\Z/2)}{H^2(X;\Z)}\right) \label{eq:grKOK}
\end{align}
In particular, we have proved the first part of Proposition~\ref{prop:KOK_Surfaces}.
\bigskip

\proof[Proof of the remaining claims of Proposition~\ref{prop:KOK_Surfaces}]
\raggedbottom
To compute \( \tKOK{2i}{(X)} \) for general \(i\), we identify this quotient with the image of
\[\eta\colon \KO^{2i}(X)\rightarrow\KO^{2i-1}(X)\]
This image can be computed at each stage of the filtration
\begin{equation*}
  \KO^j(X)= \KO^j_0(X) \supset \KO^j_1(X) \supset \dots \supset \KO^j_4(X)
\end{equation*}
To lighten the notation, we simply write \(\KO^j_k\) for \(\KO^j_k(X)\) in the following, and we write \(H^*(X)\) for the singular cohomology of \(X\) with integral coefficients.

Since we are assuming that \(X\) is at most four-dimensional, there are only three possibly non-zero differentials in the spectral sequence computing \(\KO^*(X)\). The first two are the differentials

\begin{alignat*}{2}
      &{\Sq^2}\circ\pi \colon  & H^2(X)\; \longrightarrow &\; H^4(X;\Z/2)\\
      &{\Sq^2}         \colon  & H^2(X;\Z/2)\rightarrow &\; H^4(X;\Z/2)
\end{alignat*}
on the \( E_2 \)-page, depicted in Figure~\ref{fig:KO-AHSS}. If \( E_2^{4,-2}=H^4(X;\Z/2) \) is not killed by \( \Sq^2 \), then we have a third possibly non-trivial differential on the \( E_3 \)-page:
\begin{align*}
  d_3^{1,0}\colon{H^1(X)\rightarrow \coker(\Sq^2)}
\end{align*}
The differential \(d_3^{1,-2}\) vanishes by Lemma~\ref{lem:AHSS_d3}.
\paragraph{Computation of KO\( ^2 \)/K.}
For the last stage of the filtration, we have a commutative diagram
\begin{equation*}
\xymatrix{
  {0}\ar[r] & {\KO^2_4} \ar[r]\ar[d] & {\KO^2_3} \ar[r]\ar[d]^{\eta} & {H^3(X;\KO^{-1}(\point))} \ar[r]\ar[d]^{\cong} & {0} \\
  {0}\ar[r] & {0}       \ar[r]       & {\KO^1_3} \ar@{=}[r]          & {H^3(X;\KO^{-2}(\point))} \ar[r] & {0}
  }
\end{equation*}
Thus, \( \KO^2_3 \) maps surjectively to \( \KO^1_3\cong H^3(X;\Z/2) \).
Next, we have
\begin{equation*}
\xymatrix{
  {0}\ar[r] & {\KO^2_3} \ar[r]\ar@{->>}[d]^{\eta} & {\KO^2_2} \ar[r]\ar[d]^{\eta} & {\ker({\Sq^2}\circ\pi)} \ar[r]\ar[d]^{\pi} & {0} \\
  {0}\ar[r] & {\KO^1_3} \ar[r]             & {\KO^1_2} \ar[r]       & {\ker(\Sq^2)} \ar[r] & {0}
  }
\end{equation*}
where \( \ker({\Sq^2}\circ\pi)\subset H^2(X;\KO^{0}(\point)) \) and \( \ker(\Sq^2)\subset H^2(X;\KO^{-1}(\point)) \).
This gives a short exact sequence of images, which must split since all groups involved are killed by multiplication with \( 2 \). Thus, we have
\begin{align*}
 \im\!\left(\KO^2_2\overset{\eta}\rightarrow\KO^1_2\right)
   & \cong \KO^1_3 \oplus \im\!\Big(\ker({\Sq^2}\circ\pi)\rightarrow\ker(\Sq^2)\Big)\\
   & \cong H^3(X;\Z/2) \oplus \ker(\Sq^2_{\Z})
\end{align*}
where \(\Sq^2_{\Z}\) is the composition defined at the beginning of this section.
Finally, the diagram
\begin{equation*}
\xymatrix{
 {0}\ar[r] & {\KO^2_2} \ar@{=}[r]\ar[d]^{\eta} & {\rKO^2(X)} \ar[r]\ar[d]^{\eta} & {0} \ar[r]\ar[d] & {0} \\
 {0}\ar[r] & {\KO^1_2} \ar[r]                  & {\rKO^1(X)} \ar[r]              & {E_4^{1,0}}\ar[r] & {0}
 }
\end{equation*}
shows that \( \im(\rKO^2(X)\overset{\eta}\rightarrow\rKO^1(X)) \) \( \cong \) \( \im(\KO^2_2\overset{\eta}\rightarrow\KO^1_2) \).
So \( \trKOK{2}{(X)} \) is isomorphic to \(H^3(X;\Z/2) \oplus \ker(\Sq^2_{\Z}) \), as claimed.
\paragraph{Computation of KO\( ^4 \)/K.}
Proceeding as in the previous case, we consider the diagram
\begin{equation*}
\xymatrix{
  {0} \ar[r] & {H^4(X)} \ar@{=}[r]\ar@{->>}[d] & {\rKO^4(X)} \ar[r]\ar[d]^{\eta} & {0} \ar[r]\ar[d] & {0}\\
  {0} \ar[r] & {\KO^3_4} \ar[r] & {\rKO^3(X)} \ar[r] & {H^3(X)}\ar[r] & {0}
}
\end{equation*}
where the vertical map on the left is the composition
\begin{equation*}
  H^4(X;\KO^0(\point))\twoheadrightarrow H^4(X;\KO^{-1}(\point)) \twoheadrightarrow \coker({\Sq^2}\circ\pi) = \KO^3_4 \\
\end{equation*}
Since \( {\Sq^2}\circ\pi \) has the same cokernel as \( \Sq^2_{\Z} \), we may deduce that
\begin{equation*}
\trKOK{4}{(X)} \cong \im(\rKO^4(X)\overset{\eta}\rightarrow\rKO^3(X)) \cong \coker(\Sq^2_{\Z})
\end{equation*}
\paragraph{Computation of KO\( ^6 \)/K.}
In this case, the commutative diagram
\begin{equation*}
\xymatrix{
 {0} \ar[r] & {H^2(X)} \ar@{=}[r]\ar[d] & {\rKO^6(X)} \ar[r]\ar[d]^{\eta} & {0} \ar[r]\ar[d] & {0} \\
 {0} \ar[r] & {0}                       \ar[r]           & {\rKO^5(X)} \ar@{=}[r]   & {H^1(X)} \ar[r] & {0}
}
\end{equation*}
demonstrates that the map \( \rKO^6(X)\overset{\eta}\rightarrow\rKO^5(X) \) is zero, so \( (\rKO^{6}/\K)(X) \) is trivial.
This completes the computations of \( (\rKO^{2i}/\K)(X) \).
\endproof

\encouragepagebreak{\sectionair}
\section{Comparison}\label{sec:CS:comparison}
In this section, we finally compare our two sets of results for
complex curves and surfaces and prove the comparison theorem mentioned
in the introduction. As a warm-up, we review the situation one finds
in K-theory.

\subsubsection{K-groups}
The algebraic and topological K-groups of a smooth complex curve \( C \) are given by
\begin{align*}
   \K_0(C) &= \Z\oplus\Pic(C)\\
   \K^0(C) &= \Z \oplus H^2(C;\Z)
\end{align*}
The comparison map \( \K_0(C)\rightarrow \K^0(C) \) is always surjective. Indeed, the map is an isomorphism on the first summand, and for a projective curve the map on reduced groups is simply the projection from \( \Pic(C)\cong\Z\oplus\Jac(C) \) onto the free part. It follows that the map is still surjective if we remove a finite number of points from \( C \).

For a smooth complex surface \( X \), we have seen in \eqref{eq:grKalg} and \eqref{eq:grKtop} that we have filtrations on the K-groups such that
\begin{align*}
 \gr^*(\K_0(X)) &\cong \Z \oplus \Pic(X) \oplus \CH^2(X)\\
 \gr^*(\K^0(X)) &\cong \Z \oplus H^2(X;\Z) \oplus H^4(X;\Z)
\end{align*}
Both isomorphisms can be written as \( (\rank, c_1,c_2) \), and the
comparison map \( \K_0(X)\rightarrow \K^0(X) \) corresponds to the
usual comparison maps on the filtration.
Of course, on the first summand we again have the identity.
Moreover, the map \( \CH^2(X)\rightarrow H^4(X;\Z) \) is always surjective:

If \( X \) is projective, this follows from the fact that \( H^4(X;\Z) \) is generated by a point. In general, we can embed any smooth surface \( X \) into a projective surface \( \bar{X} \) as an open subset with complement a divisor with simple normal crossings (\cf Lemma~\ref{lem:SmoothVarieties}). Then \( \CH^2(\bar{X}) \) surjects onto \( \CH^2(X) \), and similarly \( H^4(\bar{X};\Z) \) surjects onto \( H^4(X;\Z) \), so that the claim follows. Since we will need this observation in a moment, we record it as a lemma.
\begin{lemma}\label{lem:ker_c1}
For any smooth complex surface \( X \), the natural map \( \CH^2(X)\rightarrow H^4(X;\Z) \) is surjective.
\end{lemma}
\begin{corollary}\label{cor:comparison_K}
For any smooth complex surface \( X \), the natural map \( \K_0(X)\rightarrow\K^0(X) \) is surjective if and only if the natural map \( \Pic(X)\rightarrow H^2(X;\Z) \) is surjective.
\end{corollary}
When \( X \) is projective, we see from the exponential sequence that we have surjections if and only if \( X \) has geometric genus zero. Equivalently, this happens if and only if \( X \) has full Picard rank, \ie if and only if its Picard number \( \rho \) agrees with its second Betti number \( b_2 \). More generally, the Picard group of any smooth complex surface can be written as
\begin{equation}\label{eq:Pic-surface}
 \Pic(X)=\Z^{\rho}\oplus H^2(X;\Z)_{\text{tors}}\oplus \Pic^0(X)
\end{equation}
where \( \rho\leq b_2 \) is an integer that generalizes the Picard number, \( H^2(X;\Z)_{\text{tors}} \) is the torsion subgroup of \( H^2(X;\Z) \), and \( \Pic^0(X) \) is a divisible group \cite{PedriniWeibel:Surfaces}*{Corollary~6.2.1}. Again, the natural map \( \Pic(X)\rightarrow H^2(X;\Z) \) is surjective if and only if \( \rho=b_2 \).

\encouragepagebreak{\subsectionair}
\subsection{The classical Witt group}
For the classical Witt group \( \W^0(X) \), the situation can be analysed in a similar way as in the case of \( \K_0(X) \), using our description in terms of Stiefel-Whitney classes.
\begin{proposition}\label{prop:comparison_0}
 For a smooth complex curve \( C \), the map
 \begin{align*}
 && \gw^0\colon &\GW^0(C) \twoheadrightarrow \KO^0(C)\quad\text{is surjective, and}\\
 && \w^0\colon  &\W^0(C)  \overset{\cong}\rightarrow \tKOK{0}{(C)}\quad\text{is an isomorphism.}
 \end{align*}
 Similarly, for a smooth complex surface \( X \), both \( \gw^0 \) and
 \( \w^0 \) are surjective, and  \( \w^0 \) is an isomorphism if and only if \( \Pic(X) \) surjects onto \( H^2(X;\Z) \).
\end{proposition}
\begin{proof}
It suffices to show the corresponding statements for reduced groups. Let \( X \) be a smooth complex variety of dimension at most two. Since the first Stiefel-Whitney classes are always surjective, we have a row exact commutative diagram of the form
\begin{equation*}
  \xymatrix{
    {0} \ar[r]
  & {\ker(w_1)}   \ar[r]        \ar[d]
  & {\rGW^0(X)}   \ar[r]^-{w_1} \ar[d]^{\gw^0}
  & {H^1_{\et}(X;\Z/2)} \ar[r]  \ar[d]^{\cong}
  & {0}
  \\
    {0} \ar[r]
  & {\ker(w_1^{\topl})}    \ar[r]
  & {\rKO^0{(X)}}  \ar[r]^-{w_1^{\topl}}
  & {H^1(X;\Z/2)}  \ar[r]
  & {0}
  }
\end{equation*}
Similarly, by (the proof of) Corollary~\ref{cor:etale-w2_Surface} and by Corollary~\ref{cor:w2_Surface} we have a row exact commutative diagram
\begin{equation*}
  \xymatrix{
    {0} \ar[r]
  & {H^0(\ker(c_1))}\ar[r]\ar[d]
  & {\ker(w_1)}\ar[d] \ar[r]^-{w_2}
  & {H^2_{\et}(X;\Z/2)} \ar[r]\ar[d]^{\cong}
  & {0}
  \\
    {0} \ar[r]
  & {r(\ker(c_1^{\topl}))}\ar[r]
  & {\ker(w_1^{\topl})}  \ar[r]^-{w_2^{\topl}}
  & {H^2(X;\Z/2)} \ar[r]
  & {0}
  }
\end{equation*}
In both diagrams, we have written \(w_i^{\topl}\) and \(c_i^{\topl}\) for the topological characteristic classes to avoid ambiguous notation.
The kernels of \(c_1 \) and \( c_1^{\topl} \) can be identified with \( \CH^2(X) \) and \( H^4(X;\Z) \), respectively, via the second Chern classes. Thus, Lemma~\ref{lem:ker_c1} implies that the vertical map on the left of the lower diagram is a surjection. The surjectivity of the comparison map \( \gw^0 \) on \( \GW^0(X) \) follows.

If we apply the same analysis to \( \W^0(X) \), then the second diagram reduces to a commutative square
\begin{equation*}
  \xymatrix{
     {\ker(\bar{w}_1)}\ar[d] \ar[r]^-{\bar{w}_2}_-{\cong}
  &  {\quotient{H^2_{\et}(X;\Z/2)}{\Pic(X)}}\ar@{->>}[d]
  \\
     {\ker(\bar{w}_1^{\topl})}  \ar[r]^-{\bar{w}_2^{\topl}}_-{\cong}
  &  {\quotient{H^2(X;\Z/2)}{H^2(X;\Z)}}
  }
\end{equation*}
When \( X \) is a curve, the vertical arrow on the right is an isomorphism, proving the claim. In general, we have a row-exact commutative diagram of the following form:
\begin{equation*}
\xymatrix{
   {0}         \ar[r]
 & {\quotient{\Pic(X)}{2}} \ar[r] \ar@{>->}[d]
 & {H^2_{\et}(X;\Z/2)} \ar[r]\ar[d]^{\cong}
 & {\quotient{H^2_{\et}(X;\Z/2)}{\Pic(X)}} \ar[r] \ar@{->>}[d]
 & {0}
 \\ {0} \ar[r]
 & {\quotient{H^2(X;\Z)}{2}} \ar[r]
 & {H^2(X;\Z/2)} \ar[r]
 & {\quotient{H^2(X;\Z/2)}{H^2(X;\Z)}} \ar[r]
 & {0}
 }
\end{equation*}
If \( \Pic(X)\rightarrow H^2(X;\Z) \) is surjective, then the two outer maps become isomorphisms and it follows that the comparison map \( \w^0 \) is also an isomorphism. Using the description of the Picard group of \( X \) given by \eqref{eq:Pic-surface}, we see that the converse is also true.
\end{proof}

\encouragepagebreak{\subsectionair}
\subsection{Shifted Witt groups}\label{sec:CS:comparison:shifted}
We now generalize Proposition~\ref{prop:comparison_0} to shifted
groups. As indicated in the introduction and in
Remark~\ref{rem:disclaimer}, we will tacitely rely on the Standing
Assumptions~1.9 of \cite{Me:WCCV} in all that follows.  The final
result is stated in Theorems~\ref{thm:comparison_Curves} and
\ref{thm:comparison_Surfaces}.

For curves, there is in fact very little left to be shown. We nevertheless give a detailed proof in preparation for a similar line of argument in the case of surfaces.
\begin{theorem}\label{thm:comparison_Curves}
 For any smooth complex curve \( C \), the maps
\shortendisplayskip
 \begin{alignat*}{3}
 \gw^i\colon{} && \GW^i(C)&\rightarrow \KO^{2i}(C)
 &&\quad\text{are surjective, and the maps}\\
 \w^i\colon{}  &&  \W^i(C)&\rightarrow \tKOK{2i}{(C)}
 &&\quad\text{are isomorphisms.}
 \end{alignat*}
\end{theorem}
\begin{proof}
We know that the comparison maps appearing here are isomorphisms on a
point (\cf \cite{Me:WCCV}*{Section~2.2}), so the claims are equivalent to the corresponding claims involving reduced groups. It suffices to show that the maps \( \gw^i\colon \rGW^i(C)\rightarrow \rKO^{2i}(C) \) are surjective and that the maps \( \w^i\colon \rW^i(C)\rightarrow\rKO^{2i-1}(C) \) are injective.

First, suppose \( C \) is affine. Then the cohomology of \( C \) is concentrated in degrees \( 0 \) and \( 1 \) and we see that \( \W^1(C) \), \( \W^2(C) \), \( \W^3(C) \) and \( \rKO^2(X) \), \( \rKO^4(C) \), \( \rKO^6(C) \) all vanish. Thus, the claims are trivially true.

The case that \( C \) is projective can be reduced to the affine case. Indeed, if \( p \) is any point on \( C \), then \( \tilde{C}:=C-p \) is affine. By comparing the localization sequences arising from the inclusion of \( \tilde{C} \) into \( C \), we see that the comparison maps for \( C \) must also have the desired properties:

\begin{samepage}
\begin{equation}\label{seq:localizationCp}
\vcenter{\raisebox{0pt}[0pt][15pt]{
\xymatrix@C=8pt{
   {\dots} \ar[r]
 & {\GW^{i-1}(p)}\ar[r]\ar[d]^{\cong}
 & {\rGW^i(C)} \ar[r]\ar@{..>}[d]
 & {\rGW^i(\tilde{C})} \ar[r]\ar@{->>}[d]
 & {\W^{i}(p)} \ar[r]\ar[d]^{\cong}
 & {\rW^{i+1}(C)} \ar[r]\ar@{..>}[d]
 & {\rW^{i+1}(\tilde{C})}\ar[r]\ar@{>->}[d]
 & {\dots}
 \\
   {\dots} \ar[r]
 & {\KO^{2i-2}(p)} \ar[r]
 & {\rKO^{2i}(C)} \ar[r]
 & {\rKO^{2i}(\tilde{C})} \ar[r]
 & {\KO^{2i-1}(p)} \ar[r]
 & {\rKO^{2i+1}(C)} \ar[r]
 & {\rKO^{2i+1}(\tilde{C})} \ar[r]
 & {\dots}
}
}}
\end{equation}
\vspace{5pt}

\qedhere
\end{samepage}
\end{proof}
\begin{corollary}\label{cor:comparison-twist_Curves}
 Theorem~\ref{thm:comparison_Curves} also hold for groups with twists in any line bundle.
\end{corollary}
\begin{proof}
Introducing a twist by the line bundle \(\OO(p)\) into the localization sequence \eqref{seq:localizationCp} only affects the groups of \(C\), so we can conclude as before. More generally, given a line bundle \(\OO(D) \) associated with a divisor \(D=\sum n_i p_i\) on \(C\), we can similarly reduce to the case of a trivial line bundle over \(C-\bigcup_i p_i\).
\end{proof}

We now want to imitate this proof for surfaces, replacing the role of points on the curve by curves on the surface. We first prove the following.

\begin{proposition}\label{prop:comparison_Surfaces}
 For any smooth complex surface \( X \), the comparison maps have the properties indicated by the following arrows:
 \begin{equation*}
  \begin{cases}
   \gw^0\colon{\GW^0(X)\twoheadrightarrow\KO^0(X)} \\
   \;\w^1\colon{\;\W^1(X)\rightarrowtail\KO^1(X)}
  \end{cases}\\
  \begin{cases}
   \gw^2\colon{\GW^2(X)\twoheadrightarrow\KO^4(X)}\\
   \;\w^3\colon{\;\W^3(X)\rightarrowtail\KO^5(X)}
  \end{cases}
\end{equation*}
\end{proposition}

\begin{lemma}\label{lem:S_h:aff_rank0}
Proposition~\ref{prop:comparison_Surfaces} is true when \( X \) is affine and \( \Pic(X)/2 \) vanishes.
\end{lemma}
\begin{proof}
By the theorem of Andreotti and Frankel, a smooth complex affine variety of dimension \( n \) has the homotopy type of a CW complex of real dimension at most \( n \) \citelist{\cite{AndreottiFrankel},\cite{Lazarsfeld}*{3.1}}. In particular, its cohomology is concentrated in degrees \( \leq n \) . For a smooth affine surface \( X \), Pardon's spectral sequence shows immediately that \( \W^1(X) = \W^2(X) = \W^3(X) = 0 \). Similarly, the Atiyah-Hirzebruch spectral sequence shows that \( \rKO^4(X) \) vanishes. Thus, three of the four claims are trivially satisfied. The fact that \( \gw^0\colon{\GW^0(X)\twoheadrightarrow\KO^0(X)} \) is surjective was already shown in \ref{prop:comparison_0}.
\end{proof}

Before proceeding with the proof of Proposition~\ref{prop:comparison_Surfaces}, we make a note of two general facts that we will use. First, we will need the following standard consequence of Hironaka's resolution of singularities:
\begin{lemma}\label{lem:SmoothVarieties}
 In characteristic zero, any smooth variety can be embedded into a smooth compact variety with complement a divisor with simple normal crossings.
\end{lemma}
\begin{proof}
 Let \( X \) be a smooth variety over a field of characteristic zero, and let \( \bar{X} \) be some compactification. Any singularities of \( \bar{X} \) may be resolved without changing the smooth locus  \cite{Kollar}*{Theorem~3.36}, so we may assume that \( \bar{X} \) is smooth. Applying the Principalization Theorem \cite{Kollar}*{Theorem~3.26} to (the ideal sheaf of) the complement of \( X \) in \( \bar{X} \) yields the variety we are looking for.

\end{proof}

Note that in dimension two there is no distinction between compactness and projectivity: any smooth compact surface is projective \cite{Hartshorne:AmpleSub}*{II.4.2}. It follows that an arbitrary smooth surface is at least quasi-projective. Secondly, we will need the following lemma concerning generators of the Picard group.
\begin{lemma}\label{lem:PicMod2}
 Let \( X \) be a smooth quasi-projective variety over an algebraically closed field of characteristic zero. Then any element of \( \Pic(X)/2 \) can be represented by a smooth prime divisor (\ie by a smooth irreducible subvariety of codimension \( 1 \)). If \( X \) is projective, we may moreover take the divisor to be very ample (\ie to be given by a hyperplane section of \( X \) for some embedding of \( X \) into some \( \P^N \)).
\end{lemma}
\begin{proof}
 We consider the case when \( X \) is projective first. If \( X \) is a projective curve, \( \Pic(X)/2\cong\Z/2 \) is generated by a point and there is nothing to show. So we may assume \( \dim(X)\geq 2 \).

 Fix a very ample line bundle \( \lb{L} \) on \( X \).
 Any element of \( \Pic(X)/2 \) can be lifted to a line bundle \( \lb{M} \) on \( X \).
 Tensoring with any sufficiently high power of \( \lb{L} \) will yield a very ample line bundle \( \lb{M}\otimes\lb{L}^m \) \cite{Lazarsfeld}*{1.2.10}.
 In particular, if we take \( m \) large and even we obtain a very
 ample line bundle that maps to the class of \( \lb{M} \) in \(
 \Pic(X)/2 \).
 The claim now follows from Bertini's theorem on hyperplane sections in characteristic zero \cite{Hartshorne}*{Corollary~10.9 and Remark~10.9.1}.

 In general, if \( X \) is quasi-projective, we may embed it as an open subset into a smooth resolution \( \bar{X} \) of its projective closure. Then \( \Pic(\bar{X}) \) surjects onto \( \Pic(X) \), and we obtain smooth prime divisors on \( X \) generating \( \Pic(X)/2 \) by restriction.
\end{proof}

\begin{example}
Consider \( \tilde \P^2 \), the blow-up of \( \P^2 \) at a point \( p \). Its Picard group is given by \( \Pic(\tilde \P^2)=\Z[H]\oplus\Z[E] \), where \( H \) is a hyperplane section of \( \P^2 \) that misses \( p \) and \( E \) is the exceptional divisor, both isomorphic to \( \P^1 \). A divisor \( a[H]-b[E] \) is ample if and only if \( a>b>0 \). Thus, \( \{[H], 2[H]-[E]\} \) is a basis of \( \Pic(\tilde \P^2) \) consisting of ample divisors. A smooth curve representing \( 2[H]-[E] \) is given by the birational transform of a smooth conic in \( \P^2 \) through \( p \).
\end{example}

\begin{proof}[Proof of Proposition~\ref{prop:comparison_Surfaces}]
Let \( X \) be a smooth surface. By Lemma~\ref{lem:SmoothVarieties}, we can find a projective surface \( \bar{X} \) and smooth curves \( D_1  \), \dots, \( D_k \) on \( \bar{X} \) whose union \( \bigcup D_i \) is the complement of \( X \) in \( \bar{X} \). On the other hand, by Lemma~\ref{lem:PicMod2}, we can find smooth ample curves generating \( \Pic(\bar{X})/2 \). Let \( C_1 \), \dots, \( C_\rho \) be a subset of these curves generating \( \Pic(X)/2 \), and put
\begin{equation*}
  U_i:=X-C_1-C_2-\cdots - C_i
\end{equation*}
For sufficiently large \( k \), the divisors \( \sum_j D_j + k(C_1+\cdots+C_i) \) are ample on \( \bar{X} \). Thus, each \( U_i \) is affine.
Moreover, we see from the exact sequences
\begin{equation*}
  \Z[C_{i+1}|_{U_i}]\rightarrow \Pic(U_i) \rightarrow \Pic(U_{i+1}) \rightarrow 0
\end{equation*}
and the choice of the \( C_i \) that \( \rank_{\Z/2}(\Pic(U_i)/2)=\rho-i \).
Thus, by Lemma~\ref{lem:S_h:aff_rank0}, Proposition~\ref{prop:comparison_Surfaces} holds for \( U_{\rho} \).

We can now proceed as in the proof of Theorem~\ref{thm:comparison_Curves}, by adding the curves back in to obtain \( X \). Namely, consider the successive open inclusions \( U_{i+1}\hookrightarrow U_i \). The closed complements of these are given by restrictions of the curves \( C_i \), so we obtain a diagram similar to \eqref{seq:localizationCp} with \( U_i \) playing the role of \( C \), \( U_{i+1} \) in the role of \( \tilde{C} \) and \( C_{i+1} \) in the role of a point. If the normal bundle of \( C_{i+1} \) in \( U_i \) is not trivial, the sequences will in fact involve twisted groups of \( C_i \), but in any case we can conclude using Lemma~\ref{cor:comparison-twist_Curves}.
\end{proof}

\begin{corollary}
Proposition~\ref{prop:comparison_Surfaces} also holds for groups with twists in a line bundle over \( X \).
\end{corollary}
\begin{proof}
As we have seen, generators of \( \Pic(X)/2 \) can be represented by smooth curves on \( X \). Thus, we can argue as in the proof of Corollary~\ref{cor:comparison-twist_Curves}.
\end{proof}

\begin{theorem}\label{thm:comparison_Surfaces}
Suppose \( X \) is a smooth complex surface for which the natural map \( \Pic(X)\rightarrow H^2(X;\Z) \) is surjective. Then the comparison maps
\begin{alignat*}{3}
 \gw^i&\colon  &\GW^i(X)& \rightarrow \KO^{2i}(X)
 &&\quad\text{are surjective, and the maps}\\
  \w^i&\colon  & \W^i(X)& \rightarrow \tKOK{2i}{(X)}
 &&\quad\text{are isomorphisms.}
 \end{alignat*}
\end{theorem}
\noindent By Proposition~\ref{prop:comparison_0}, the assumption on the map \(\Pic(X)\rightarrow H^2(X;\Z)\) is clearly necessary.

\begin{proof}
Consider the squaring operations \( S^1 \) and \( \Sq^2_{\Z} \) appearing in the computations of the Witt and \( \tKOK{}{} \)-groups. For any smooth complex variety \( X \), we have a commutative diagram
\begin{equation*}
 \xymatrix@R=6pt{
  {\Pic(X)/2} \ar[r]^-{S^1}         \ar@{^{(}->}[d]       & {\CH^2(X)/2}  \ar[d]\\
  {H^2(X;\Z)/2}   \ar[r]            \ar@{^{(}->}[d]       & {H^4(X;\Z)/2}    \ar@{^{(}->}[d]\\
  {H^2(X;\Z/2)} \ar[r]^-{\Sq^2}                           & {H^4(X;\Z/2)}
 }
\end{equation*}
When \( X \) is a surface, the two vertical maps on the right are both
isomorphisms, and the horizontal map in the middle is essentially \(
\Sq^2_{\Z} \).  So whenever \( \Pic(X) \) surjects onto \( H^2(X;\Z)
\), we may identify \( S^1 \) and \( \Sq^2_{\Z} \). It follows by
comparison of Propositions~\ref{thm:W_Surface} and
\ref{prop:KOK_Surfaces} that the Witt groups of \( X \) agree with the groups \( \tKOK{2i}{(X)} \). On the other hand, we see from Proposition~\ref{prop:comparison_Surfaces} that each of the maps \( \w^i\colon{\W^i(X)\rightarrow\tKOK{2i}{(X)}} \) is either surjective or injective. Given that we are dealing with finite groups, these maps must be isomorphisms. Moreover, since we know from Corollary~\ref{cor:comparison_K} that we also have a surjection from the algebraic to the topological K-group of \( X \), we may deduce via the Karoubi/Bott sequences that the maps \( \gw^i\colon{\GW^i(X)\rightarrow\KO^{2i}(X)} \) are surjective for all values of \( i \).
\end{proof}

\begin{example}\label{eg:non-cellular-Surfaces} The
  result discussed here is completely independent of the comparison result in
  \cite{Me:WCCV}. In particular, there are lots of projective surfaces
  whose geometric genus $\rho_g$ is zero but which are not
  cellular. We briefly list a few concrete examples. In each case, non-cellularity may be deduced from the fact that the fundamental group does not vanish. All data is freely quoted from \cite{Beauville}.
\begin{itemize}
\item Any surface ruled over a curve $C$ has geometric genus zero, but since $b_1(X)=g(C)$ it cannot be ruled unless it is rational.
\item Enriques surfaces have $\rho_g=0$ but fundamental group $\Z/2$. Similarly, $\rho_g=0$ for the Godeaux surface, but its fundamental group is $\Z/5$.
\item Bielliptic surfaces have $\rho_g=0$ but first Betti number $b_1=1$.
\end{itemize}
\end{example}

\subsection{\texorpdfstring{$\Z/2$}{\textbf{Z}/2}-coefficients}
\label{sec:comparison:homotopic:2}
In this final section, we examine how the integral comparison
result discussed in this article is related to comparison results for
the corresponding theories with \( \Z/2 \)-coefficients. In brief, it
turns out that ``the integral comparison isomorphism for Witt groups'' is closely
related to ``comparison isomorphisms for all higher
hermitian K-groups with \( \Z/2 \)-coefficients''. This is made precise in Proposition~\ref{prop:integral-vs-2-comparison} and Corollary~\ref{cor:my-comparison-with-2}.

Recall from \cite{Me:WCCV}*{Section~2} that the comparison maps considered so
far may be viewed as low-degree versions of comparison maps from
higher (hermitian) K-groups to (real) topological K-groups:
\begin{alignat*}{2}
       k_i &\colon &\K_i(X)      &\rightarrow\K^{-i}(X) \\
       k_h^{p,q} &\colon &\Kh^{p,q}(X) &\rightarrow\KO^p(X)
\end{alignat*}
That is, we may identify the map \( k \) with \(k_0 \), \( \gw^i \)
with \( k_h^{2i,i} \) and \( \w^i \) with \( k_h^{2i-1,i-1} \).

The definition of these higher maps may be generalized to take
care of coefficients. In particular, we may consider algebraic and
hermitian K-groups with \(\Z/2\)-coefficients and obtain comparison maps
\begin{alignat*}{2}
       k_i&\colon &\K_i(X;\Z/2)      &\rightarrow\K^{-i}(X;\Z/2) \\
  k_h^{p,q}&\colon &\Kh^{p,q}(X;\Z/2) &\rightarrow\KO^p(X;\Z/2)
\end{alignat*}
These share all the formal properties of their integral counterparts
as listed in \cite{Me:WCCV}*{Section~2}. For example, the existence of Karoubi and Bott sequences with \(\Z/2\)-coefficients may be deduced from the \(3\times 3\)-Lemma in triangulated categories \cite{May:Additivity}*{Lemma~2.6}, and the comparison maps are again compatible with these.

The behaviour of the comparison maps for \( \K \)-theory with
\(\Z/2\)-coefficients was predicted by the Quillen-Lichtenbaum
conjectures: they are isomorphisms in all degrees \( i\geq{\dim(X)-1}
\) and injective in degree \( \dim(X)-2 \). Proofs may be found in
\cite{Levine:BlochLichtenbaum}*{Corollary~13.5 and Remark~13.2} and
\cite{Voevodsky:MC-with-2}*{Theorem~7.10}. In particular, on a complex
point or curve we have isomorphisms in all non-negative degrees, while
for a smooth complex surface \( X \) we have isomorphisms in positive degrees and an inclusion in degree zero:
\begin{align*}
 \phantom{\text{for all \( i>0 \)}}\quad
 \K_i(X;\Z/2)&\xrightarrow{\;\cong\;} \K^{-i}(X;\Z/2)
 \quad \text{for all \( i>0 \)}\\
 \K_0(X;\Z/2)&\xhookrightarrow{\;\phantom{\cong}\;} \K^0(X;\Z/2)
\end{align*}
For these low-dimensional cases, proofs may also be found in \citelist{\cite{Suslin:K-and-MC}*{4.7}\cite{PedriniWeibel:Surfaces}*{Theorem~2.2}}.

The Quillen-Lichtenbaum conjecture for a complex point implies an analogous statement for the hermitian comparison maps. Namely, it is not difficult to see that the comparison maps on Witt groups with \(\Z/2\)-coefficients
\begin{equation*}
  \w^i\colon \W^i(X;\Z/2)\rightarrow \tKOK{2i}{(X;\Z/2)}
\end{equation*}
are also isomorphisms when \(X\) is a point. (For example, this can be deduced from Proposition~\ref{prop:integral-vs-2-comparison} below.) Thus, the following statement may be obtained via Karoubi-induction.
\begin{corollary}
The comparison maps
\begin{alignat*}{2}
 k_h^{p,q}\colon \Kh^{p,q}(\point;\Z/2)  \rightarrow\KO^p(\point;\Z/2)
\end{alignat*}
are isomorphisms in all non-negative degrees, \ie for all \( (p,q) \) with \( 2q-p\geq 0 \).
\end{corollary}

When \( X \) is a surface, we have seen that the integral comparison
maps on Witt groups are isomorphisms whenever \( \Pic(X) \) surjects
onto \( H^2(X;\Z) \). By passing to \( \Z/2 \)-coefficients, we may
obtain isomorphisms on the level of \GrothendieckWitt groups under
just one additional topological constraint on \( X \).
\begin{proposition}\label{prop:integral-vs-2-comparison}
 Let \( X \) be a smooth complex variety, of any dimension. If the odd topological K-groups of \( X \) contain no 2-torsion (\ie if \(\K^1(X)[2]=0\)), then the integral comparison maps
\begin{align*}
 \W^i(X)&\rightarrow\tKOK{2i}{(X)}
\intertext{%
are isomorphisms for all \( i \) if and only if the comparison maps with \(\Z/2\)-coefficients
}
 \W^i(X;\Z/2)&\rightarrow\tKOK{2i}{(X;\Z/2)}
\end{align*}
are isomorphisms for all \( i \).
\end{proposition}
Before giving the proof, we make one preliminary observation.
\begin{lemma}
Let \( X \) be a topological space such that \(\K^1(X)[2]=0\). Then multiplication by \( \eta\in\KO^{-1}(\point) \) induces an isomorphism
       \[
       \tKOK{2i}{(X)}\xrightarrow[\cong]{\;\eta\;}\KO^{2i-1}(X)[2]
       \]
\end{lemma}
\begin{proof}
In general, since \( 2\eta=0 \), the following exact sequence may be extracted from the Bott sequence:
\begin{equation*}
 0\rightarrow\tKOK{2i}{(X)}\xrightarrow{\;\eta\;} \KO^{2i-1}(X)[2]\rightarrow\K^1(X)[2]
\end{equation*}
This proves the claim.
\end{proof}
\begin{proof}[Proof of Proposition~\ref{prop:integral-vs-2-comparison}]
We claim that we have the following row-exact commutative diagram:
\begin{equation}\label{diag:prop:integral-vs-2-comparison}
\tag{\textasteriskcentered}
\vcenter{\xymatrix{
   {0}                                 \ar[r]
 & {\frac{\GW^i(X)}{\K_0(X)}}          \ar[r]\ar[d]
 & {\frac{\GW^i(X;\Z/2)}{\K_0(X;\Z/2)}}\ar[r]\ar[d]
 & {\W^{i+1}(X)}                       \ar[r]\ar[d]
 & {0} \\
   {0}                                     \ar[r]
 & {\frac{\KO^{2i}(X)}{\K^0(X)}}           \ar[r]
 & {\frac{\KO^{2i}(X;\Z/2)}{\K^0(X;\Z/2)}} \ar[r]
 & {\KO^{2i+1}(X)[2]}                      \ar[r]
 & {0}
}}
\end{equation}
Indeed, the lower exact row may be obtained by applying the Snake Lemma to the following diagram of short exact sequences induced by the Bockstein sequences for \( \K \)- and \( \KO \)-theory:
\begin{equation*}
\xymatrix{
   {0}           \ar[r]
 & {\K^0(X)/2}   \ar[r]^-{\cong}\ar[d]
 & {\K^0(X;\Z/2)}\ar[r]\ar[d]
 & {0}                 \ar[d] \\
   {0}\ar[r]
 & {\KO^{2i}(X)/2}   \ar[r]
 & {\KO^{2i}(X;\Z/2)}\ar[r]
 & {\KO^{2i+1}(X)[2]}  \ar[r]
 & {0}
}
\end{equation*}
The upper row of \eqref{diag:prop:integral-vs-2-comparison} may be obtained similarly from the Bockstein sequences for algebraic and hermitian K-theory. The vertical maps are induced by the comparison maps in degrees \( 0 \), \( 0 \) and \( -1 \), respectively. Using the canonical identification of \( \GW^i(X)/\K_0(X) \) with \( \W^i(X) \), we may however identify the first vertical map with the usual comparison map \( \w^i \) in degree \( -1 \). Likewise, the second vertical map may be identified with the comparison map \( \w^i \) for Witt groups with \( \Z/2 \)-coefficients. Lastly, by the previous lemma, the entry in the lower right corner may be identified with \( \tKOK{2i+2}{(X)} \). Thus, diagram \eqref{diag:prop:integral-vs-2-comparison} can be rewritten in a form from which both implications of the proposition may be deduced:
\begin{equation*}
\raisebox{\totalheight}{\xymatrix{
   {0}                \ar[r]
 & {\W^i(X)}          \ar[r]\ar[d]^-{\w^i}
 & {\W^i(X;\Z/2)}     \ar[r]\ar[d]^-{\w^i}
 & {\W^{i+1}(X)}      \ar[r]\ar[d]^-{\w^{i+1}}
 & {0} \\
   {0}                       \ar[r]
 & {\KOK{2i}{(X)}}           \ar[r]
 & {\KOK{2i}{(X;\Z/2)}}      \ar[r]
 & {\KOK{2i+2}{(X)}}         \ar[r]
 & {0}
}}%
\qedhere
\end{equation*}
\end{proof}

\begin{lemma}
Let \( X \) be a smooth complex variety of dimension at most two. The map \( \K_0(X;\Z/2)\hookrightarrow \K^0(X;\Z/2) \) is an isomorphism if and only if the map \( \Pic(X)\rightarrow H^2(X;\Z) \) is surjective and \( \K^1(X)[2]=0 \).
\end{lemma}
\begin{proof}
We see from the description of the Picard group \eqref{eq:Pic-surface} that \( \Pic(X)\rightarrow H^2(X;\Z) \) is surjective if and only if it is surjective after tensoring with \(\Z/2\). Moreover, by a similar argument as in Corollary~\ref{cor:comparison_K}, the surjectivity of the map \( \Pic(X)/2\rightarrow H^2(X;\Z)/2 \) is equivalent to the surjectivity of the map \( \K_0(X)/2\rightarrow \K^0(X)/2 \).  The claim follows from these equivalences and a commutative diagram induced by the Bockstein sequences:
\begin{equation*}
\raisebox{\totalheight}{\xymatrix{
   {0}            \ar[r]
 & {\K_0(X)/2}    \ar[r]^{\cong}\ar[d]
 & {\K_0(X;\Z/2)} \ar[r]\ar[d]
 & {0}            \ar[d]
 \\
   {0}            \ar[r]
 & {\K^0(X)/2}    \ar[r]
 & {\K^0(X;\Z/2)}  \ar[r]
 & {\K^1(X)[2]}   \ar[r]
 & {0}
}}%
\qedhere
\end{equation*}
\end{proof}
If we combine the known results for K-theory with our results for Witt groups and Proposition~\ref{prop:integral-vs-2-comparison}, we obtain the following corollary via Karoubi induction.
\begin{corollary}\label{cor:my-comparison-with-2}
Let \( X \) be a smooth complex variety of dimension at most two. Assume that the natural map \( \Pic(X)\rightarrow H^2(X;\Z) \) is surjective and that \( \K^1(X) \) has no 2-torsion. Then the hermitian comparison maps
\begin{flalign*}
   \qquad && \Kh^{p,q}(X;\Z/2)&\longrightarrow\KO^p(X;\Z/2) &&
\intertext{%
are isomorphisms in all non-negative degrees, \ie for all \( (p,q) \) such that \(2q-p\geq 0\). In particular, for all shifts \( i \) we have isomorphisms
}
          && \GW^i(X;\Z/2)&\overset{\cong}\longrightarrow\KO^{2i}(X;\Z/2) &&
\end{flalign*}
\end{corollary}
For example, the conditions of the corollary are satisfied for any smooth complex curve, and for any simply-connected projective surface of geometric genus zero. The condition that \(\K^1(X)[2]=0 \) can be rephrased in terms of the integral cohomology group \( H^3(X;\Z) \):
\begin{lemma}
Let \( X \) be a smooth complex variety of dimension at most two. Then the torsion of \( \K^1(X) \) agrees with the torsion of \(  H^3(X;\Z) \).
\end{lemma}
\begin{proof}
By Lemma~\ref{lem:AHSS_d3}, the Atiyah-Hirzebruch spectral sequence for the K-theory of \( X \) collapses. Since the first integral cohomology group \( H^1(X;\Z) \) of a smooth complex curve or surface is free, we find that \( \K^1(X)\cong H^1(X;\Z)\oplus H^3(X;\Z) \), and the lemma follows.
\end{proof}
\begin{remark*}
Conversely, if we assume the analogue of the Quillen-Lichtenbaum conjecture for hermitian K-theory, \ie if we assume that the hermitian comparison maps with \( \Z/2 \)-coefficients are isomorphisms in high degrees, then we can recover our comparison theorem for Witt groups for all surfaces \( X \) with \( \K^1(X)[2]=0 \). Said analogue appeared recently in \cite{BKOS:QuillenLichtenbaum}. However, it does not seem possible to relate our result to the Quillen-Lichtenbaum conjecture for surfaces with \( 2\)-torsion in \(\K^1(X)\). Such surfaces do exist. In particular, if \( X \) is an Enriques surface, then \( \Pic(X) \) surjects onto \( H^2(X;\Z) \) but \( \K^1(X)[2]\cong\pi_1(X)\cong\Z/2 \) \cite{Beauville}*{page~90}.
\end{remark*}

\subsection*{Acknowledgements}
The results presented here stem from the author's PhD studies, which were generously funded by the Engineering and Physical Sciences Research Council (EPSRC). Part of this work was moreover supported by a research grant of the Cambridge Philosophical Society. I warmly thank my supervisor Burt Totaro for his continuing support, and for encouraging me to put this part of my research down on paper. Jens Hornbostel has supplied helpful comments on an earlier version of this paper.

\encouragepagebreak{\subsectionair}


\begin{bibdiv}
\begin{biblist}

\bib{AndreottiFrankel}{article}{
  author={Andreotti, Aldo},
  author={Frankel, Theodore},
  title={The Lefschetz theorem on hyperplane sections},
  journal={Ann. of Math. (2)},
  volume={69},
  date={1959},
  pages={713--717},
}

\bib{Arason}{article}{
  author={Arason, J{\'o}n Kristinn},
  title={Der Wittring projektiver R\"aume},
  journal={Math. Ann.},
  volume={253},
  date={1980},
  number={3},
  pages={205--212},
  issn={0025-5831},
}

\bib{Auel:MilnorRemarks}{article}{
  author={Auel, Asher},
  title={Remarks on the Milnor conjecture over schemes},
  date={2011},
  note={\tt arXiv:1109.3294},
}
\bib{Auel:Clifford-invariant}{article}{
  author={Auel, Asher},
  title={Surjectivity of the total Clifford invariant and Brauer dimension},
  date={2011},
  note={\tt arXiv:1108.5728},
}

\bib{Balmer:TWGI}{article}{
  author={Balmer, Paul},
  title={Triangular Witt groups. I. The 12-term localization exact sequence},
  journal={$K$-Theory},
  volume={19},
  date={2000},
  number={4},
  pages={311--363},
  issn={0920-3036},
}

\bib{Balmer:TWGII}{article}{
  author={Balmer, Paul},
  title={Triangular Witt groups. II. From usual to derived},
  journal={Math. Z.},
  volume={236},
  date={2001},
  number={2},
  pages={351--382},
  issn={0025-5874},
}

\bib{Balmer:Handbook}{article}{
  author={Balmer, Paul},
  title={Witt groups},
  conference={ title={Handbook of K-theory. Vol. 1, 2}, },
  book={ publisher={Springer}, place={Berlin}, },
  date={2005},
  pages={539--576},
}

\bib{BalmerWalter:GWSS}{article}{
  author={Balmer, Paul},
  author={Walter, Charles},
  title={A Gersten-Witt spectral sequence for regular schemes},
  journal={Ann. Sci. \'Ecole Norm. Sup. (4)},
  volume={35},
  date={2002},
  number={1},
  pages={127--152},
  issn={0012-9593},
}

\bib{BGPW:Gersten}{article}{
  author={Balmer, Paul},
  author={Gille, Stefan},
  author={Panin, Ivan},
  author={Walter, Charles},
  title={The Gersten conjecture for Witt groups in the equicharacteristic case},
  journal={Doc. Math.},
  volume={7},
  date={2002},
  pages={203--217 (electronic)},
}

\bib{Beauville}{book}{
  author={Beauville, Arnaud},
  title={Complex algebraic surfaces},
  series={London Mathematical Society Student Texts},
  volume={34},
  edition={2},
  publisher={Cambridge University Press},
  place={Cambridge},
  date={1996},
}

\bib{BKOS:QuillenLichtenbaum}{article}{
  author={Berrick, Jon},
  author={Karoubi, Max},
  author={{\O }stv{\ae }r, Paul Arne},
  author={Schlichting, Marco},
  title={The homotopy fixed point theorem and the Quillen-Lichtenbaum conjecture in hermitian K-theory},
  date={2011},
  note={\tt arxiv:1011.4977v2},
}

\bib{BlochOgus}{article}{
  author={Bloch, Spencer},
  author={Ogus, Arthur},
  title={Gersten's conjecture and the homology of schemes},
  journal={Ann. Sci. {\'E}cole Norm. Sup. (4)},
  volume={7},
  date={1974},
  pages={181--201 (1975)},
}

\bib{Brosnan:Steenrod}{article}{
  author={Brosnan, Patrick},
  title={Steenrod operations in Chow theory},
  journal={Trans. Amer. Math. Soc.},
  volume={355},
  date={2003},
  number={5},
  pages={1869--1903},
  issn={0002-9947},
}

\bib{Delzant}{article}{
  author={Delzant, Antoine},
  title={D\'efinition des classes de Stiefel-Whitney d'un module quadratique sur un corps de caract\'eristique diff\'erente de $2$},
  journal={C. R. Acad. Sci. Paris},
  volume={255},
  date={1962},
}

\bib{ElmanWadsworth}{article}{
   author={Elman, Richard},
   author={Wadsworth, Adrian R.},
   title={Hereditarily quadratically closed fields},
   journal={J. Algebra},
   volume={111},
   date={1987},
   number={2},
   pages={475--482},
}

\bib{EKV}{article}{
  author={Esnault, H{\'e}l{\`e}ne},
  author={Kahn, Bruno},
  author={Viehweg, Eckart},
  title={Coverings with odd ramification and Stiefel-Whitney classes},
  journal={J. Reine Angew. Math.},
  volume={441},
  date={1993},
  pages={145--188},
}
\bib{Fasel:IntersectionFormula}{article}{
   author={Fasel, Jean},
   title={The excess intersection formula for Grothendieck-Witt groups},
   journal={Manuscripta Math.},
   volume={130},
   date={2009},
   number={4},
   pages={411--423},
   issn={0025-2611},
}
\bib{FaselSrinivas:Chow-Witt}{article}{
   author={Fasel, Jean},
   author={Srinivas, Vasudevan},
   title={Chow-Witt groups and Grothendieck-Witt groups of regular schemes},
   journal={Adv. Math.},
   volume={221},
   date={2009},
   number={1},
   pages={302--329},
   issn={0001-8708},
}
\bib{Fujii:P}{article}{
  author={Fujii, Michikazu},
  title={$K\sb {0}$-groups of projective spaces},
  journal={Osaka J. Math.},
  volume={4},
  date={1967},
  pages={141--149},
  issn={0030-6126},
}

\bib{Fulton:Intersection}{book}{
  author={Fulton, William},
  title={Intersection theory},
  series={Ergebnisse der Mathematik und ihrer Grenzgebiete. 3.~Folge}
  volume={2},
  edition={2},
  publisher={Springer-Verlag},
  place={Berlin},
  date={1998},
}

\bib{Fernandez}{article}{
  author={Fern{\'a}ndez-Carmena, Fernando},
  title={The Witt group of a smooth complex surface},
  journal={Math. Ann.},
  volume={277},
  date={1987},
  number={3},
  pages={469--481},
}

\bib{GelfandManin}{book}{
  author={Gelfand, Sergei},
  author={Manin, Yuri},
  title={Methods of homological algebra},
  series={Springer Monographs in Mathematics},
  edition={2},
  publisher={Springer-Verlag},
  place={Berlin},
  date={2003},
  pages={xx+372},
}

\bib{Gille:gradedGW}{article}{
   author={Gille, Stefan},
   title={A graded Gersten-Witt complex for schemes with a dualizing complex
   and the Chow group},
   journal={J. Pure Appl. Algebra},
   volume={208},
   date={2007},
   number={2},
   pages={391--419},
}

\bib{Hartshorne:AmpleSub}{book}{
  author={Hartshorne, Robin},
  title={Ample subvarieties of algebraic varieties},
  series={Notes written in collaboration with C. Musili. Lecture Notes in Mathematics, Vol. 156},
  publisher={Springer-Verlag},
  place={Berlin},
  date={1970},
  pages={xiv+256},
}

\bib{Hartshorne}{book}{
  author={Hartshorne, Robin},
  title={Algebraic geometry},
  note={Graduate Texts in Mathematics, No. 52},
  publisher={Springer-Verlag},
  place={New York},
  date={1977},
  pages={xvi+496},
}

\bib{HJJS:BasicBundleTheory}{book}{
  author={Husem{\"o}ller, Dale},
  author={Joachim, Michael},
  author={Jur{\v {c}}o, Branislav},
  author={Schottenloher, Martin},
  title={Basic bundle theory and $K$-cohomology invariants},
  series={Lecture Notes in Physics},
  volume={726},
  publisher={Springer},
  place={Berlin},
  date={2008},
}
\bib{Knebusch:curves}{article}{
   author={Knebusch, Manfred},
   title={On algebraic curves over real closed fields. II},
   journal={Math. Z.},
   volume={151},
   date={1976},
   number={2},
   pages={189--205},
   issn={0025-5874},
}
\bib{Knebusch:varieties}{article}{
  author={Knebusch, Manfred},
  title={Symmetric bilinear forms over algebraic varieties},
  book={ publisher={Queen's Univ.}, place={Kingston, Ont.}, },
  date={1977},
  pages={103--283. Queen's Papers in Pure and Appl. Math., No. 46},
  note={Also available from \url {epub.uni-regensburg.de/12783/}},
}

\bib{Kollar}{book}{
  author={Koll{\'a}r, J{\'a}nos},
  title={Lectures on resolution of singularities},
  series={Annals of Mathematics Studies},
  volume={166},
  publisher={Princeton University Press},
  place={Princeton, NJ},
  date={2007},
  pages={vi+208},
}

\bib{Laborde}{article}{
  author={Laborde, Olivier},
  title={Classes de Stiefel-Whitney en cohomologie \'etale},
  note={Colloque sur les Formes Quadratiques (Montpellier, 1975)},
  journal={Bull. Soc. Math. France Suppl. Mem.},
  number={48},
  date={1976},
  pages={47--51},
}

\bib{Lazarsfeld}{book}{
  author={Lazarsfeld, Robert},
  title={Positivity in algebraic geometry. I},
  series={Ergebnisse der Mathematik und ihrer Grenzgebiete. 3.~Folge.}
  volume={48},
  note={Classical setting: line bundles and linear series},
  publisher={Springer-Verlag},
  place={Berlin},
  date={2004},
  pages={xviii+387},
}

\bib{Levine:BlochLichtenbaum}{article}{
  author={Levine, Marc},
  title={K-theory and motivic cohomology of schemes},
  note={Preprint},
  eprint={www.math.uiuc.edu/K-theory/0336/},
  date={1999},
}
\bib{Levine:Slice-and-GW}{article}{
   author={Levine, Marc},
   title={The slice filtration and Grothendieck-Witt groups},
   journal={Pure Appl. Math. Q.},
   volume={7},
   date={2011},
   number={4, Special Issue: In memory of Eckart Viehweg},
   pages={1543--1584},
   issn={1558-8599},
}
\bib{May:Additivity}{article}{
  author={May, J. Peter},
  title={The additivity of traces in triangulated categories},
  journal={Adv. Math.},
  volume={163},
  date={2001},
  number={1},
  pages={34--73},
}
\bib{Merkurjev:MC}{article}{
  author={Merkurjev, Alexander},
  title={On the norm residue symbol of degree $2$},
  journal={Dokl. Akad. Nauk SSSR},
  volume={261},
  date={1981},
  number={3},
  pages={542--547},
}

\bib{Milne:LEC}{article}{
  author={Milne, James},
  title={Lectures on Etale Cohomology (v2.10)},
  year={2008},
  eprint={www.jmilne.org/math/},
}

\bib{Milnor:Conjectures}{article}{
  author={Milnor, John},
  title={Algebraic $K$-theory and quadratic forms},
  journal={Invent. Math.},
  volume={9},
  date={1969/1970},
  pages={318--344},
}

\bib{MilnorStasheff}{book}{
  author={Milnor, John},
  author={Stasheff, James},
  title={Characteristic classes},
  note={Annals of Mathematics Studies, No. 76},
  publisher={Princeton University Press},
  place={Princeton, N. J.},
  date={1974},
  pages={vii+331},
}

\bib{OVV:MilnorConjecture}{article}{
  author={Orlov, Dmitri},
  author={Vishik, Alexander},
  author={Voevodsky, Vladimir},
  title={An exact sequence for $K^M_\ast /2$ with applications to quadratic forms},
  journal={Ann. of Math. (2)},
  volume={165},
  date={2007},
  number={1},
  pages={1--13},
}

\bib{Pardon}{article}{
  author={Pardon, William},
  title={The filtered Gersten-Witt resolution for regular schemes},
  note={Preprint},
  eprint={www.math.uiuc.edu/K-theory/0419/},
  date={2004},
}
\bib{Parimala:curves-local}{article}{
   author={Parimala, Raman},
   title={Witt groups of curves over local fields},
   journal={Comm. Algebra},
   volume={17},
   date={1989},
   number={11},
   pages={2857--2863},
   issn={0092-7872},
}
\bib{PedriniWeibel:Surfaces}{article}{
  author={Pedrini, Claudio},
  author={Weibel, Charles},
  title={The higher $K$-theory of a complex surface},
  journal={Compositio Math.},
  volume={129},
  date={2001},
  number={3},
  pages={239--271},
  issn={0010-437X},
}
\bib{Schlichting:MayerVietoris}{article}{
   author={Schlichting, Marco},
   title={The Mayer-Vietoris principle for Grothendieck-Witt groups of
   schemes},
   journal={Invent. Math.},
   volume={179},
   date={2010},
   number={2},
   pages={349--433},
   issn={0020-9910},
}
\bib{Schlichting:GWnotes}{article}{
   author={Schlichting, Marco},
   title={Hermitian K-theory, derived equivalences and Karoubi's fundamental theorem},
   note={unpublished draft},
   date={2007},
}
\bib{Sujatha:RPS}{article}{
   author={Sujatha, Ramdorai},
   title={Witt groups of real projective surfaces},
   journal={Math. Ann.},
   volume={288},
   date={1990},
   number={1},
   pages={89--101},
   issn={0025-5831},
}
\bib{Suslin:K-and-MC}{article}{
  author={Suslin, Andrei},
  title={Algebraic $K$-theory and motivic cohomology},
  conference={ title={ 2}, address={Z\"urich}, date={1994}, },
  book={ publisher={Birkh\"auser}, place={Basel}, },
  date={1995},
}

\bib{Totaro:Witt}{article}{
  author={Totaro, Burt},
  title={Non-injectivity of the map from the Witt group of a variety to the Witt group of its function field},
  journal={J. Inst. Math. Jussieu},
  volume={2},
  date={2003},
  number={3},
  pages={483--493},
  issn={1474-7480},
}

\bib{Voevodsky:RPO}{article}{
  author={Voevodsky, Vladimir},
  title={Reduced power operations in motivic cohomology},
  journal={Publ. Math. Inst. Hautes {\'E}tudes Sci.},
  number={98},
  date={2003},
  pages={1--57},
  issn={0073-8301},
}

\bib{Voevodsky:MC-with-2}{article}{
  author={Voevodsky, Vladimir},
  title={Motivic cohomology with ${\bf Z}/2$-coefficients},
  journal={Publ. Math. Inst. Hautes {\'E}tudes Sci.},
  number={98},
  date={2003},
  pages={59--104},
}

\bib{Walter:TGW}{article}{
  author={Walter, Charles},
  title={Grothendieck-Witt groups of triangulated categories},
  note={Preprint},
  eprint={www.math.uiuc.edu/K-theory/0643/},
  date={2003},
}

\bib{Walter:PB}{article}{
  author={Walter, Charles},
  title={Grothendieck-Witt groups of projective bundles},
  note={Preprint},
  eprint={www.math.uiuc.edu/K-theory/0644/},
  date={2003},
}

\bib{Me:WCCV}{article}{
  author={Zibrowius, Marcus},
  title={Witt groups of complex cellular varieties},
  journal={Documenta Math.},
  number={16},
  date={2011},
  pages={465--511},
}

\bib{Me:Thesis}{thesis}{
  author={Zibrowius, Marcus},
  title={Witt Groups of Complex Varieties},
  type={PhD Thesis},
  organization={University of Cambridge},
  date={2011},
  note={Available at \url{www.dspace.cam.ac.uk/handle/1810/239413}.},
}

\end{biblist}
\end{bibdiv}
\end{document}